\documentclass[12pt,a4paper]{amsart}

\usepackage[utf8]{inputenc}
\DeclareUnicodeCharacter{00A0}{}

\usepackage[T1]{fontenc}
\usepackage{lmodern}

\usepackage{amssymb,xspace}
\usepackage{hyperref}

 \let\mathsec\mathsf
 \IfFileExists{mathrsfs.sty}
  {\usepackage{mathrsfs}\let\mathcal\mathscr}
  {\let\mathscr\mathcal}

\allowdisplaybreaks[3]

\theoremstyle{plain}
\newtheorem{prop}[subsubsection]{Proposition}

\newtheorem{theo}[subsubsection]{Theorem}
\newtheorem{coro}[subsubsection]{Corollary}

\newtheorem{lemm}[subsubsection]{Lemma}

\theoremstyle{definition}

\theoremstyle{remark}
\newtheorem{rema}[subsubsection]{Remark}

 \makeatletter

\def\subsection{\@startsection{subsection}{2}%
   \z@{.7\linespacing\@plus.3\linespacing}{0.3\linespacing}
   {\normalfont\bfseries}}

 \let\c@equation\c@subsubsection
 \let\cl@equation\cl@subsubsection

\def\l@table{\@tocline{0}{3pt plus2pt}{0pt}{}{\itshape}}

\makeatother

\def\ie{\emph{i.e.}\xspace}

\def\Card{\mathop\#}
\def\Cl{\mathord{\mathscr C\!\ell}}
\def\Cl{\mathscr C}
\def\Clan{\mathscr C^{\text{\upshape an}}}

\def\AD{{\mathbb A}}

\def\C{{\mathbf C}}

\def\ga{{\mathbf G}_a}
\def\gm{{\mathbf G}_m}
\def\N{{\mathbf N}}
\def\P{{\mathbf P}}
\def\Q{{\mathbf Q}}
\def\R{{\mathbf R}}

\def\Z{{\mathbf Z}}

\def\Zeta{{\mathrm Z}}

\def\Tube{{\mathsf T}}

\let\ra\rightarrow

\let\epsilon\varepsilon \let\eps\epsilon
\let\epsilon\varepsilon
\let\phi\varphi
\let\emptyset\varnothing
\let\leq\leqslant
\let\geq\geqslant
\let\le\leq
\let\ge\geq

\def\eff{{\text{\upshape eff}}}
\def\abs#1{\left\lvert{#1}\right\rvert}
\def\norm#1{\left\|{#1}\right\|}

\def\DeclareMathOperator#1#2{\def #1{\operatorname{#2}}}
 
\DeclareMathOperator{\Re}{Re} 
 
\DeclareMathOperator{\Pic}{Pic}
\DeclareMathOperator{\Gal}{Gal}

\DeclareMathOperator{\Spec}{Spec}

\DeclareMathOperator{\rang}{rank}

\DeclareMathOperator{\div}{div}
\DeclareMathOperator{\Val}{Val}

\DeclareMathOperator{\Val}{Val}

\DeclareMathOperator{\Ind}{Ind}

\def\bmA{{\bar{\mathscr A}}}

\def\loccit{\emph{loc.\ cit.}\xspace}

\def\resp{\emph{resp.}\xspace}
\def\eg{\emph{e.g.}\xspace}
\def\ie{\emph{i.e.}\xspace}
	\def\Tr{\operatorname{Tr}}
	\def\EP{\operatorname{EP}}
	\def\Nm{\operatorname{N}}

\title[Integral points of bounded height]{Integral points of bounded height \\
on partial equivariant compactifications \\ of vector groups}

\author{Antoine Chambert-Loir}
\address{Universit\'e de Rennes~1, IRMAR--UMR 6625 du CNRS, Campus de Beaulieu, 35042 Rennes Cedex, France}
\address{Institut universitaire de France}
\email{antoine.chambert-loir@univ-rennes1.fr}

\author{Yuri Tschinkel}
\address{Courant Institute, NYU, 251 Mercer St.  New York, NY 10012, USA}
\email{tschinkel@cims.nyu.edu}

\makeatletter
 \@input{volumes.aux}
\makeatother

\begin{document}
\date{\today}
 
\begin{abstract}
We establish asymptotic formulas  for the number of integral points
of bounded height on partial equivariant compactifications of vector
groups.

\end{abstract}


\maketitle


\section{Introduction}

In this paper we study the distribution of integral points of
bounded height on partial equivariant compactifications of additive
groups, \ie, quasi-projective algebraic varieties defined over number 
fields, equipped with an action of a vector group $G\simeq\mathbb G_a^n$, 
and containing $G$ as an open dense orbit.

The case of projective compactifications
has been the subject of \cite{chambert-loir-t2002}.
Motivated by a conjecture of Manin, 
we established analytic properties of corresponding height zeta 
functions and deduced
asymptotic formulas for the number of \emph{rational} points of
bounded height. 
Examples of such varieties are
projective spaces of any dimension and their blow-ups 
along a reduced subscheme contained in a hyperplane, as well as
certain singular Del Pezzo surfaces 
(listed in~\cite{derenthal-loughran2009}).

The study of \emph{integral} solutions of diophantine equations
is a major branch of number theory. Our aim was to extend to
the theory of \emph{integral} points the 
geometric framework and techniques
developed in the context of Manin's conjecture. 

Important ingredients in the context of rational points
on a smooth projective variety~$X$ over a number field~$F$ are:
\begin{itemize}
\item the convex geometry of the cone of effective divisors in the 
Picard group of $X$;
\item the positivity of the anticanonical class of~$X$,
which is assumed to be ample ($X$ Fano), or at least big;
\item Tamagawa measures on the adelic space~$X(\AD_F)$ of~$X$. 
\end{itemize}
Manin's conjecture~\cite{franke-m-t89}, as refined by Peyre~\cite{peyre95}
and~\cite{batyrev-t95b,batyrev-t98},
predicts that the number of $F$-rational points
of anticanonical height~$\leq B$ is asymptotic to
\[ \alpha(X) \beta(X) \tau(X)  B (\log B)^{b-1}, \]
where $b=\rang(\Pic(X))$ is the rank of the Picard group of $X$, 
$\alpha(X)$ is a rational number depending
on the location of the anticanonical class in 
the effective cone, $\beta(X)$ is the cardinality of the Galois cohomology
group $\mathrm H^1(\Pic(\bar X))$, and $\tau(X)$
is the volume of the closure of~$X(F)$ in~$X(\AD_F)$
with respect to the Tamagawa measure derived from
the metrization of the canonical line bundle defining
the height.

In \cite{chambert-loir-tschinkel2009a}, 
we generalized the theory of Tamagawa measures
to the quasi-projective case and established asymptotic formulas for
volumes of analytic and adelic height balls of growing radius. 

In the present paper,
we establish analytic properties of height zeta functions
for integral points on partial equivariant compactifications~$U$
of vector groups. We assume that $U$ is the complement
to a $G$-invariant divisor~$D$ in a smooth projective
equivariant compactification~$X$ of~$G$, and that, geometrically,
$D$ has simple normal crossings. We fix integral models for~$X$
(assumed to be proper over~$\Spec(\mathfrak o_F)$), $D$
and~$U$, and consider the subset
of $S$-integral points in~$X(F)$ with respect to these models,
$S$ being a finite
set of places of~$F$, including all archimedean places.
This will be denoted by $\mathscr U(\mathfrak o_{F,S})$.

In this setup we defined in~\cite{chambert-loir-tschinkel2009a}, for each place~$v\in S$,
a simplicial complex $\Clan_{F_v}(D)$, which we called
the \emph{analytic Clemens complex}, and which encodes the incidence
properties of the $v$-adic manifolds 
given by the irreducible components of~$D$.
Its dimension is equal to the maximal number of irreducible components
of~$D$ defined over the local field~$F_v$ whose intersection
has $F_v$-points,  minus one.

By definition, the log-canonical line bundle on~$X$ is $K_X+D$.
The log-anticanonical line bundle is its opposite; 
in the case under study, it is big.
An adelic metrization of this line bundle
defines a height function~$H$ on $X(F)$ which satisfies a finiteness
property: for any real number~$B$, the set of points~$x$ in~$G(F)$
such that $H(x) \leq B$  is finite.
The formula proved in Theorem~\ref{theo.main} asserts that 
the number $N(B)$ of points in~$G(F)\cap \mathscr U(\mathfrak o_{F,S})$
of bounded log-anticanonical height~$\leq B$ satisfies
\[ N(B) \sim \Theta B (\log B)^{b-1}, \]
where 
\[ b = \rang(\Pic(U)) + \sum_{v\in S} (1+\dim \Clan_{F_v}(D)) \]
and $\Theta$ is a positive real number: 
a product of an adelic volume, involving places outside~$S$, and contributions
from places in~$S$ given by sums of integrals over the minimal dimensional
strata of the analytic Clemens complex.

The case of rational points corresponds to the case $D=\emptyset$.
Provided the chosen model of~$D$ is also empty,
rational points and integral points coincide then, \ie, 
$\mathscr U(\mathfrak o_{F,S})=X(F)$, and
the simplicial complex $\Clan_{F_v}(D)$ has dimension~$-1$.
We then recover the asymptotic distribution of rational points 
which we had established in~\cite{chambert-loir-t2002}.

The notions of $S$-integral points and heights extend to the adelic
group $G(\AD_F)$. In a very general context,
we had proved in~\cite{chambert-loir-tschinkel2009a}
an asymptotic formula for the volume~$V(B)$
of $S$-integral adelic points
in~$G(\AD_F)$ of log-anticanonical height bounded by~$B$.
Our main result says that,  for $B\ra\infty$,
\[ N(B) \sim V(B), \]
see Remark~\ref{rema:N-V}.

The explicit description of the constant~$\Theta$ as a product
of an adelic volume, and of local volumes,
implies an equidistribution property of $S$-integral points
for some explicit measure on $X(\AD_F)$
defined as the product of a measure on the adelic
space $X(\AD_F^S)$ (where $\AD_F^S$ is the ring analogous to the adeles of~$F$,
but where only places not belonging to~$S$
are taken into account), and of measures on 
the $v$-adic spaces $X(F_v)$, for all places $v\in S$.
On $X(\AD_F^S)$, it is given as 
a regularized measure on the adelic space~$U(\AD_F^S)$,
in the spirit of the definition of Tamagawa measures on algebraic
groups~\cite{weil82}, or on projective varieties~\cite{peyre95}.
At places~$v\in S$, it is localized on the
strata of minimal dimension of the $F_v$-analytic Clemens complex.
(See Section~\ref{sec:equidistribution}.)

\bigskip

The proof relies on a strategy developed in our previous papers
on rational points. We introduce the height zeta function
\[ \mathrm Z(s) = \sum_{x\in G(F)\cap \mathscr U(\mathfrak o_{F,S})} H(x)^{-s}, \]
for a complex parameter~$s$. We prove its convergence
for $\Re(s)>1$, and then establish its meromorphic continuation 
in some half-plane $\Re(s)>1-\delta$ whose poles are
located on finitely many arithmetic progressions on the line $\Re(s)=1$.
The pole of highest order, equal to~$b$, is at $s=1$. The asymptotic
formula for~$N(B)$ follows from an appropriate Tauberian theorem.

Observe that the restriction of the height~$H$ to~$G(F)$
can be written as a product over all places~$v$ of~$F$
of a local height function~$H_v$ defined on~$G(F_v)$.
Hence, we can extend the height as a function on~$G(\AD_F)$.
Let us also introduce, for each place~$v\not\in S$, the characteristic
function~$\delta_v$ of the set~$\mathscr U(\mathfrak o_v)$
of $v$-integral points 
in~$X(F_v)$ and set $\delta_v\equiv 1$ for $v\in S$.
Using the Poisson summation formula for the locally compact group $G(\AD_F)$
and its cocompact discrete subgroup $G(F)$,
we can write
\[ \mathrm  Z(s) = \sum_{x\in G(F)}  \prod_{v\in\Val(F)} \delta_v(x) H_v(x)^{-s}
 = \sum_{\psi\in (G(\AD_F)/G(F))^*} \prod_v \mathscr F_v( \delta_v H_v^{-s};\psi_v), \]
where, for each place~$v$ of~$F$,
\[ \mathscr F_v( \delta_vH_v^{-s};\psi_v) = \int_{G(F_v)} \delta_v(x) H_v(x)^{-s} \psi_v(x)\,\mathrm dx_v \]
is the local Fourier transform at the local component~$\psi_v$
of the global character~$\psi$.

The Fourier transform
at the trivial character furnishes the main term in the right
hand side of the Poisson formula.
Its analytic properties have been established in~\cite{chambert-loir-tschinkel2009a}.

In the case of rational points, the local integrals
$\mathscr F_v(H_v^{-s};\psi_v)$
converge absolutely when $s$ belongs to the half-plane $\{\Re(s)>0\}$
and the  poles of the global Fourier transform are all
accounted for by the Euler products.
The main technical difficulty which appears for integral points
is that we need to establish, for places $v\in S$,
some meromorphic continuation to the left of~$\Re(s)=1$ 
of the local integrals~$\mathscr F_v(\delta_v H_v^{-s};\psi_v)$.
Moreover, we need to provide decay estimates
sufficient to ensure the convergence of the right hand side
in some half-plane~$\Re(s)>1-\delta$.
  
For places in~$S$, the local integrals at the trivial character
are integrals of Igusa-type. Such integrals were 
investigated in our previous paper~\cite{chambert-loir-tschinkel2009a}: 
using Tauberian theorems,
the asymptotic of local height balls was deduced from their
analytic properties.
For non-trivial characters, we are led to consider \emph{oscillatory}
Igusa-type integrals, higher-dimensional analogs of integrals
of the form
\[ \int_{F_v} \abs x^{s-1} \psi(a x^d) \Phi(x)\,\mathrm dx, \]
where $d\in\Z$, $a\in F_v$, and $\Phi$ is a Schwartz function
on~$F_v$.
Establishing their meromorphic
continuation, with appropriate decay in the parameter~$\psi$,
is the technical heart of this work. This is accomplished 
in Section~\ref{sec.non-trivial-S}, using estimates
for stationary phase integrals from Section~\ref{sect:dim1}.
We note that similar inequalities have already played an important
rôle in the question of counting integral points of
bounded height,  \eg, in the adelic presentation by~\cite{lachaud1982}
of the circle method.
We refer also to~\cite{cluckers2011} for recent general developments
in the same spirit.
 
Ultimately, the corresponding Euler products at non-trivial
characters do not contribute to the main pole
of the height zeta function.  However, we found that this property
was more subtle than in the case of rational points
(see Section~\ref{sec.leading-pole}).
Precisely at this point, we use the
hypothesis that the height function is the log-anticanonical one.

\bigskip

\emph{Acknowledgments. ---}
During the final stages of preparation of this work, 
we have benefited from conversations with E. Kowalski. 
We also thank N.~Katz for his interest and stimulating comments.

The first author was supported by the Institut
universitaire de France, as well as by the National Science
Foundation under agreement No.~DMS-0635607. He would also
like to thank the Institute for Advanced Study in Princeton
for its warm hospitality which permitted the completion of this paper.
The second author was partially supported by NSF grants DMS-0739380 and 
0901777.

We thank the referees for their comments which helped to improve
the exposition.

\section{One-dimensional theory}
\label{sect:dim1}

\subsection{Local theory: characters and quasi-characters}
\label{sect:local}

Let $F$ be a local field of characteristic zero, \ie, a completion of a number field with respect to a valuation.   
It is locally compact and we denote by  $\mu$ a Haar measure on
the additive group~$F$. 
We write $\abs\cdot_F$ for the corresponding modulus function on~$F$,
defined by the equality 
$\mu(a \Omega)=\abs a\mu(\Omega)$, for any bounded open subset~$\Omega$
of~$F$. For $F\neq\C$, this is an absolute value on~$F$; if $F=\C$,
it is the square of the usual absolute value. 
When there is no risk of confusion, we shall remove the index~$F$.
We write $\Tr$ and $\Nm$ for the absolute trace and norm, these take values in 
the corresponding completion of $\Q$.
For non-archimedean $F$, let $\mathfrak o_F$ be the ring of integers, 
$\mathfrak o_F^*$ its multiplicative group, $\mathfrak d_F$ the local different of~$F$, and $q$ the order of the residue field. For archimedean~$F$,
we write $\mathfrak o_F^*$ for the subgroup of
elements of absolute value one.  
From now on we fix a local Haar measure~$\mu$ on~$F$, normalized so that
\[ 
\int_{\abs{x}\le 1} \mathrm d\mu(x)  = \begin{cases}  2 & \text{ if $F$ is real; } \\   
                                      2\pi  & \text{ if $F$ is complex; } \\
                                 \Nm(\mathfrak d_F)^{-1/2}& \text{ if $F$ is non-archimedean.}\end{cases} 
\]
We will often write $\mathrm d x$ instead of $\mathrm d\mu(x)$.

Let $\psi$ be the unimodular character
of~$F$ with respect to~$\mathrm d x$.
Concretely, 
\[
\psi(x)= \begin{cases}
 e^{2 \pi \mathrm i \Tr(x) } & \text{for non-archimedean fields~$F$,}\\
 e^{-2 \pi \mathrm i \Tr(x) } & \text{otherwise.}\end{cases}
\]
(For non-archimedean fields, the exponential is defined by embedding
naturally~$\Q_p/\Z_p$ into~$\Q/\Z$ and then applying $e^{2\pi\mathrm i\cdot}$.)
The map
\[  F\times  F\ra\C^* , \quad  (a, x)\mapsto \psi(ax) 
\]
is a perfect pairing. 
The local Fourier transform is the function on~$F$ defined by 
\[ 
\hat{f}(a) = \int_{F} f(x) \psi(ax)\mathrm d x,
\]
whenever the integral converges.

Let  $\mathscr Q(F^*)$ be the set of quasi-characters of $F^*$, \ie, 
continuous homomorphisms $\chi\colon F^*\ra\C^*$.
Examples are given by the \emph{unramified characters,}
\ie, those whose restriction to~$\mathfrak o_F^*$ is trivial;
they are of the form $x\mapsto \abs{x}^s$, for some complex
number $s\in \C$.

The image of the absolute value is equal to~$\R_{>0}$ if $F$ is archimedean,
and to~$q^\Z$ if $F$ is non-archimedean; moreover, 
the morphism of groups $\abs\cdot\colon F^*\ra \abs{F^*}$ has a section,
given by $t\mapsto t$ if $F=\R$, $t\mapsto \sqrt t$ if $F=\C$,
and $q\mapsto \varpi^{-1}$, where $\varpi$ is any fixed
uniformizing element when $F$ is non-archimedean.
Using this section, we identify the group~$F^*$ with the
product~$\abs{F^*}\times\mathfrak o_F^*$, and write
any quasi-character~$\chi$ as
\[ \chi(x)=\abs x^{s} \tilde\chi(\tilde x), \qquad x=(\abs x,\tilde x). \]
The complex number $s=s(\chi)$ 
is determined uniquely by $\chi$ if $F$ is archimedean, 
and uniquely modulo $2\pi \mathrm i /\log(q)$, if $F$ is non-archimedean.  
As in~\cite{tate67b}, we call the \emph{exponent} of
$\chi$ the real number~$r=r(\chi) = \Re(s(\chi))$.
A quasi-character is a character of $F^*$ if and only if its exponent is zero.

We say that two quasi-characters are \emph{equivalent} 
if their quotient is unramified. 
By the map $\chi\mapsto s(\chi)$, 
the set of equivalence classes of quasi-characters is viewed 
as the Riemann surface $\C$ (if $F$ is archimedean)
or $\C/(2\pi\mathrm i/\log(q))\Z)$ (if $F$ is non-archimedean).
The space $\mathscr Q(F^*)$ is a trivial bundle over this surface,
with discrete topology on the fibers; 
we endow it with the corresponding structure of a Riemann surface.

Let us assume that $F$ is non-archimedean.
Let $n$ be the minimal natural number such that 
$\chi$ is trivial on the subgroup $(1+\varpi^n \mathfrak o_F)$
of~$F^*$. 
The $\mathfrak o_F$-ideal generated by $\varpi^n$ 
is called the \emph{conductor} of $\chi$.

\subsection{Local Tate integrals}
\label{subsec:local-tate}

We now fix a Haar measure on the multiplicative group~$F^*$:
\[ 
\mathrm d^\times x = \begin{cases}
     \frac{\mathrm d x}{\abs{x}} & \text{ if $F$ is archimedean} \\
    (1-\frac{1}{q})^{-1} \frac{\mathrm dx}{\abs{x}} &  
\text{ if $F$ is non-archimedean} \end{cases}
\]
When $F$ is non-archimedean, this normalization implies
\[ 
\int_{\mathfrak o^*_F} \mathrm d^\times x = \Nm(\mathfrak d_F)^{-1/2}.
\] 

\medskip

Let $\mathscr T(F)$ be the class of functions 
$\Phi \colon F\ra \C$ satisfying the following properties:
\begin{itemize}
\item $\Phi$ and $\hat{\Phi}$ 
are continuous and absolutely integrable over $F$ with respect to $\mathrm d x$;
\item for all $r>0$ the functions $x\mapsto \Phi(x)\abs{x}^{r}$ and 
$x\mapsto \hat{\Phi}(x)\abs{x}^r$ are absolutely integrable 
over $F^*$ with respect to $\mathrm d^\times x$. 
\end{itemize}
For $\Phi\in \mathscr T(F)$ and $\chi$ with $r(\chi)> 0$ the integral
\[ 
\zeta(\Phi,\chi) = \int_{F} \Phi(x) \chi(x) \mathrm d^\times x,
\] 
is absolutely convergent. Moreover, it defines a holomorphic function on the set of equivalence classes of 
quasi-characters with $r(\chi)>0$. This function is called the local $\zeta$-function. 

As a classical  example of such integrals, we define, for 
any complex number~$s$ such that $\Re(s)>0$,
\[  \zeta_F(s) = \zeta(\mathbf 1_{\mathfrak o_F},\abs\cdot^s) 
= \int_{\abs x\leq 1} \abs x^{s}\,\mathrm d^\times x. \]
Explicitly (see~\cite{tate67b}, p.~344):
\[ \zeta_F(s) = 
\begin{cases}  2/s & \text{if $F=\R$;} \\
 2\pi/s & \text{if $F=\C$;} \\
 \mathrm N(\mathfrak d_F)^{-1/2}/ (1-q^{-s}) & \text{if $F$ is non-archimedean.}
\end{cases} \]
The corresponding residue is denoted by
\[\mathrm  c_F = \lim_{s\ra 0} s \int_{\abs{x}\leq 1} \abs{x}^s \,\mathrm
d^\times x
= \begin{cases} 2 & \text{if $F=\R$;} \\
 2\pi & \text{if $F=\C$;} \\
 \mathrm N(\mathfrak d_F)^{-1/2} / \log q & \text{otherwise.}\end{cases}
\]
These constants will appear in the definition of residue measures below.

\medskip

The main theorem of the local theory in \cite[Theorem 2.4.1]{tate67b} is:

\begin{prop}
\label{prop:tate-local}
There exists a meromorphic function~$\rho$ on~$\mathscr Q(F^*)$ such
that for any function~$\Phi\in \mathscr T(F)$, 
the local $\zeta$-function  has a meromorphic 
continuation to the domain of all 
quasi-characters given by the functional equation
\[ 
\zeta(\Phi, \chi) = \rho(\chi) \zeta(\hat{\Phi}, \hat{\chi})
\] 
where $\hat{\chi}(x)=\abs{x}\chi(x)^{-1}$. 
\end{prop}

We recall that in~\cite{tate67b}, 
the function $\rho$ is 
defined for exponents $r(\chi)\in (0,1)$ by the functional equation 
(it is holomorphic in that domain) and extended by meromorphic continuation
for all quasi-characters.
It is also computed explicitly in that paper. 
As a consequence, we obtain the following corollary.

\begin{coro}
The local $\zeta$-function is holomorphic at any character
different from the trivial character.  
Moreover, if  $\chi$ is the quasi-character $x\mapsto \abs x^s$, for $s\ra 0$, then
\[ \Phi(0) = 
  \lim_{s\ra 0} \zeta(\Phi,\abs{\, \cdot \,}^s) /  \zeta_F(s)  . \]
\end{coro}

%
%

\subsection{Oscillatory integrals}

Let $F$ be a local field.
In this section we study integrals of the form
\[ 
\mathscr I( \Phi, a, d, \chi) = \int_{F} \psi(ax^d)\chi(x)  \Phi(x)\mathrm d x, 
\quad d\in \Z, \quad a\in F, \quad \chi \in \mathscr Q(F^*), 
\] 
for suitable test functions~$\Phi$ on~$F$. (We choose to use
the measure~$\mathrm dx$; using the multiplicative Haar measure~$\mathrm d^\times x$ would result in a shift in the quasi-character~$\chi$.)
When $d$ is nonnegative, the behavior of such integrals as a function
of~$\chi$ is explained by Proposition~\ref{prop:tate-local}.
We need to understand the
decay of such integrals when the absolute value 
of the parameter~$a$ grows to infinity, for positive integers $d$. 
This result will also be used to establish a meromorphic continuation
of these integrals, with respect to~$\chi$, when
$d$ is negative.

\medskip

We recall some terminology.
A complex valued function on~$F$ is called \emph{smooth} if it is infinitely
differentiable, or locally constant, according to whether
$F$ is archimedean or not. \emph{Schwartz functions} are
smooth functions with compact support. The vector space~$\mathscr S(F)$
of Schwartz functions has a natural topology; a subset
of~$\mathscr S(F)$ is bounded when:
\begin{enumerate}
\item the supports
of all of its elements are contained in a common compact subset;
\item if $F=\R$, then for each integer~$n$, their $n$th derivatives
are uniformly bounded;
\item if $F=\C$, then for all~$(m,n)\in\N^2$, their partial
derivatives~$\partial^m_x\partial^n_y$ are uniformly bounded;
\item if $F$ is ultrametric, then there exists a positive real number~$\delta$ such that all of its elements are constant on any ball of radius~$\delta$ in~$F$.
\end{enumerate}

When the test function~$\Phi$ vanishes in a neighborhood of~$0$
and belongs to some Schwartz class,
the integral $\mathscr I$ is a \emph{non-stationary phase} integral
and its decay with respect to the parameter~$a$ is classical.
For example, the integral
\[ 
 \int_{\mathfrak o_F^*} \chi(x)\psi(ax) \, \mathrm d^\times x,
\] 
can be computed explicitly 
(see, e.g.,  \cite[p.~20--21]{weil1971});
it vanishes for $\abs a$ large enough with respect
to the conductor~$\mathfrak f(\chi)$.
%
%

We will use the following lemma.

\begin{lemm}\label{lemme2}
Let us assume that $F$ is ultrametric and let $\varpi$ be a uniformizing
element.  
Let $d$ be a positive integer; set $c=\log_q\#(\mathfrak o_F/(d\mathfrak d_F))$.
Let $\Phi$ be a Schwartz function on~$F$ with support in~$\mathfrak o_F$.
Let $n(\Phi)$ be the least positive integer~$n$ such that $\Phi$
is constant on residue classes modulo~$\varpi^n$.
For any $ a \in F$ such that 
\[ \abs a \geq \max( q^{n(\Phi)+c+1}, q^{2(c+1)} ), \]
one has
\[ \int_{\mathfrak o_F\setminus(\varpi)} \Phi(x)\psi( a  x^d)\,\mathrm dx=0.\]
\end{lemm}
\begin{proof}
Let $n$ be any integer such that $n\geq n(\Phi)$; we can write
\[ \int_{\mathfrak o_F\setminus(\varpi)}
\Phi(x)\psi(ax^d)\,\mathrm dx =\sum_{\substack{\xi \mod\varpi^n\\ \xi\neq 0\pmod\varpi}}
 \Phi(\xi) \int_{\xi+\varpi^n\mathfrak o_F} \psi(a x^d)\,\mathrm dx. \]
We shall prove in a moment that each term in this sum is zero
provided $q^{n+c}<\abs a\leq q^{2n}$.
Let us then set $\abs a=q^m$; 
these conditions
($ n+c+1\leq m\leq 2n$ and $n\geq n(\Phi)$)
are equivalent to ($ \frac12 m \leq n\leq m-c-1$ and $n\geq n(\Phi)$).
There exists such an integer~$n$ if and only if $n=m-c-1$  is a solution,
\ie, $m\geq 2(c+1)$ and $m\geq n(\Phi)+c+1$.

It remains to establish that for any $\xi\in\mathfrak o_F\setminus(\varpi)$,
the integral
\[ I= \int_{\xi+\varpi^n\mathfrak o_F} \psi( a  x^d)\,\mathrm dx \]
vanishes if $q^{n+c}<\abs a\leq q^{2n}$.
Making the change of variables $x=\xi(1+\varpi^n u)$, we obtain
\[ I=q^{-n} \int_{\mathfrak o_F}  \psi( a \xi^d(1+\varpi^n u)^d )\,\mathrm du.\]
Without loss of generality, we can also assume $\xi=1$, replacing
$ a $ by~$ a \xi^d$.
Now,
\[  a (1+\varpi^n u)^d= a   +\binom d1  a \varpi^n u
+\binom d2 a  \varpi^{2n}u^2
+ \dots +\binom dd a  \varpi^{dn}u^d. \]
Assume, as in the statement of the lemma,
that $q^{n+c}<\abs a\leq q^{2n}$.
Then, $ a \varpi^{2n}$ belongs to~$\mathfrak o_F$,
and all terms from the third one on are elements of~$\mathfrak o_F$,
so that
\[ \psi( a (1+\varpi^n u)^d)=\psi( a  ) \psi( d  a \varpi^n u)\]
for any $u\in\mathfrak o_F$. In that case,
\[ I=q^{-n} \psi( a ) \int_{\mathfrak o_F} \psi(d a \varpi^n u)\,\mathrm du. \]
This is the integral of an additive character of~$\mathfrak o_F$,
so vanishes if and only if this character is non-trivial.
By definition of the different, this happens precisely
if $d a \varpi^n\not\in\mathfrak d_F^{-1}$, which is true
since 
$\abs a  >q^n/\abs{d\mathfrak d_F}=q^{n+c}$.
Consequently, $I=0$, as claimed.
\end{proof}

\begin{prop}\label{lemm.weyl-s}
Let $F$ be a local field and $\Phi$ a 
Schwartz function on~$F$. Let $d$ be a positive integer.
For  any complex number~$s$ such that $\sigma =\Re(s)>0$, let
\[ \kappa(s)=\min\big(\frac 12,\frac{\sigma }d\big). \]
Then, as as $\abs a\ra \infty$ and $\sigma >0$, we have
\[\mathscr I(\Phi, a, d, s)  = 
\int_F \abs x_F^{s-1} \psi(a x^d) \Phi(x)\,\mathrm dx 
\ll \zeta_F(\sigma ) \min\big(1,\abs a_F^{-\kappa(s)}\big).\]
The bound is uniform when~$\Phi$ belongs to a fixed bounded subset
of the space of Schwartz functions and $\sigma >0$ is bounded from above.
\end{prop}

\begin{rema}
In fact, except when $F=\C$,
we prove the proposition with $\kappa(s)$ replaced by $\min(1,\sigma /d)$. 
\end{rema}

\begin{proof}
To simplify, we put $I(a)=\mathscr I(\Phi,a,d,s)$.
We may assume that $\Phi$ is real-valued.
By a change of variables, we also assume that $\Phi$ is zero outside of the unit ball.

\medskip

\emph{The case $F=\R$.} 
We have $\psi(x)=\exp(-2\pi \mathrm ix)$.
We introduce a parameter~$\eps>0$ and split the integral:
\[
I(a) = \int_{-\eps}^{\eps} \abs x^{s-1} \psi(ax^d) \Phi(x)\,\mathrm dx
 + \int_{\abs x\geq\eps} \abs x^{s-1} \psi(ax^d) \Phi(x)\,\mathrm dx
.\]
The first integral is bounded from above as
\[ \abs{ \int_{-\eps}^{\eps} \abs x^{s-1} \psi(ax^d) \Phi(x)\,\mathrm dx}
\leq \eps^{\sigma } \norm{\Phi}_\infty \zeta_\R(\sigma ). \]
Integration by parts in the second integral yields
\begin{align*} \int_{\eps}^\infty x^{s-1}\psi(ax^d)\Phi(x)\,\mathrm dx
& = \int_{\eps}^\infty x^{s-d}\Phi(x)\,  x^{d-1}\psi(ax^d) \,\mathrm dx \\
& = \left[ \frac{-1}{2\pi\mathrm i d a}  x^{s-d}\Phi(x) \psi(ax^d) \right]_{\eps}^\infty \\
 + \frac1{2\pi\mathrm i d a} & \int_{\eps}^\infty x^{s-d-1}  \left( (s-d)\Phi(x)+x\Phi'(x)\right)
  \psi(ax^d)\mathrm dx.
\end{align*}
Consequently, its absolute value is bounded from above by
\begin{multline*}
 \frac1{2\pi d \abs a} \left(\eps^{\sigma -d}+1\right) \norm\Phi_\infty
         + \frac1{2\pi d \abs a} \int_{\eps}^1 x^{\sigma -d-1} \left( \abs{\sigma -d}\abs\Phi+ x\abs{\Phi'}\right)\,
\mathrm dx 
\\
\ll (2\pi d\abs a)^{-1} \left( \eps^{\sigma-d} \norm\Phi_\infty + \norm\Phi_\infty
                  \eps^{\sigma-d} (\norm\Phi_\infty +\norm{\Phi'}_\infty) \right). \end{multline*}
We have a similar upper-bound for the integral from~$-\infty$ to~$-\eps$.

Fix $\eps$ so that $\eps^d \abs a=1$. Adding the obtained estimates, we have
\begin{align*}
 \abs{I(a)} & \ll
       \left(\abs{a}^{-1}+\abs{a}^{-\sigma /d}\right) 
             \zeta_\R(\sigma ) 
             \left( \norm\Phi_\infty + \norm{\Phi'}_\infty \right) \\
& \ll \abs a^{-\kappa(s)} 
             \zeta_\R(\sigma ) 
             \left( \norm\Phi_\infty + \norm{\Phi'}_\infty \right) 
\end{align*}
where the constant understood under~$\ll$ is absolute.

\medskip

\emph{The case $F=\C$.}
Note that in this case, the modulus~$\abs\cdot_{\C}$
is the square of the usual absolute value, which is
used to bound $I(a)$ from above.
Recall that
\[
\psi(u)=\exp(-2\mathrm i\pi \Re(u))\quad \text{ for  } \quad u\in\C;
\]
we write $a=\omega\exp(\mathrm i\alpha)$ with
$\omega=\abs a$ and $\alpha\in\R$. Similarly, put
$z=r\exp(\mathrm i\theta)$, so that $\abs z_\C=r^2$ and
\[ \psi(az^d) = \exp\big(-2\mathrm i\pi \omega r^d \cos(d\theta+\alpha)\big). \]
Using polar coordinates, we have
\[ I(a)=\int_0^{2\pi} \mathrm d\theta \int_0^\infty \exp(2\mathrm i\pi\omega r^d \cos(d\theta+\alpha)) \Phi(r\exp(\mathrm i\theta))\, r^{2s-1}\,\mathrm dr. \]
We write $I(a;\theta)$ for the inner integral; 
the result we proved for $F=\R$  (with~$s$ replaced by~$2s$)
implies that
\[ \abs{I(a;\theta)}  \ll   \omega^{-2\sigma /d} \abs{\cos(d\theta+\alpha)}^{-2\sigma /d}  \zeta_\R(2\sigma) ;\]
moreover, the trivial inequality $\abs{I(a;\theta)}\ll \zeta_\R(2\sigma)$ holds.
We now integrate the better of these upper bounds over~$\theta\in[0;2\pi]$.
The angles~$\theta$ such that $\abs{\cos(d\theta+\alpha)}\leq 1/\omega$ 
form a union of intervals of lengths~$\approx 1/\omega$; 
the integral over these will be $\ll 1/\omega$. 
When $\omega\ra\infty$, 
the integral over the remaining angles grows as 
$\omega^{-2\sigma /d} \int_{1/\omega}^1 u^{-2\sigma /d}\,\mathrm du$, that is:
\[ \omega^{-2\sigma /d} \max(1,\omega^{-1+2\sigma /d}) = \max(\omega^{-2\sigma /d}, \omega^{-1}). \]
Finally, since $\zeta_\R(2\sigma)=1/\sigma = \zeta_\C(\sigma)/2\pi$,
we get 
\[ \abs{I(a)} \ll \zeta_\C(\sigma)\abs  a^{-\min(1,2\sigma /d)}=\zeta_\C(\sigma) \abs a_\C^{-\kappa(s)}, \]
as claimed.

\medskip

\emph{The case when $F$ is non-archimedean.}
Let $\varpi$ be a generator
of the maximal ideal of~$\mathfrak o_F$ and let $q=\abs{\varpi}^{-1}
=\Card(\mathfrak o_F/(\varpi))$.
As above, we assume that $\Phi$ is real-valued and that its support
is contained in~$\mathfrak o_F$. Let $n(\Phi)$ be
the least positive integer~$n$ such that $\Phi$ is constant 
on residue classes modulo~$\varpi^n$.
For any nonnegative integer~$k$, let $\Phi_k$ be the Schwartz
function on~$\mathfrak o_F$ defined by $\Phi_k(u)=\Phi(\varpi^ku)$;
one has $n(\Phi_k)=\max(n(\Phi)-k,1)$.
For some nonnegative integer~$K$ (to be chosen later), we write
\begin{align*}
 I(a) & =\int_F\abs x^{s-1}\psi(ax^d)\Phi(x)\,\mathrm dx \\
& =\sum_{k=0}^{K-1}  \int_{(\varpi^k)\setminus(\varpi^{k+1})}\! \!\!\!
\abs x^{s-1}\psi(ax^d)\Phi(x)\,\mathrm dx
+\!\int_{(\varpi^K )} \abs x^{s-1}\psi(ax^d)\Phi(x)\,\mathrm dx\\
& =\sum_{k=0}^{K-1} q^{-sk}
\int_{\mathfrak o_F\setminus(\varpi)}
 \psi(a\varpi^{kd} x^d) \Phi_k(x)\,\mathrm dx \\
&\quad\quad\quad\quad \quad\quad\quad\quad\quad\quad\quad\quad 
+ q^{-sK} \int_{\mathfrak o_F}\abs x^{s-1} \psi(a\varpi^{Kd} x^d) \Phi_K(x)\,\mathrm dx.\\
\end{align*}
Let $m$ be such that $\abs a=q^m$.
By Lemma~\ref{lemme2} above, the term corresponding to the index~$k$
vanishes if 
\[ \abs {a\varpi^{kd}} \geq \max(q^{2c+2},q^{n(\Phi_k)+c+1}), \]
that is if 
\[ m-kd \geq \max(2c+2,n(\Phi)-k+c+1,c+2)=\max(2c+2,n(\Phi)+c+1-k). \]
Since $d\geq 1$, the right hand side decreases slower than left hand side,
and this holds for all $k\in\{0,\dots,K-1\}$ if it holds for~$k=K-1$,
that is if
\[ m-(K-1)d\geq 2c+2\quad\text{and}\quad m-(K-1)d\geq n(\Phi)+c+1-(K-1), \]
in other words
\[ (K-1) d \leq m-2c-2 \quad\text{and}\quad (K-1) (d-1)\leq m-n(\Phi)-c-1.\]
We choose $K$ to be the largest integer satisfying these two
inequalities, namely
\[ K= 1+\min\left( \left\lfloor \frac{m-2c-2}d\right\rfloor,
          \left\lfloor \frac{m-n(\Phi)-c-1}{d-1}\right\rfloor \right). \]
This integer is positive when $m\geq\min(2c+2,n(\Phi)+c+1)$.
In that case,
\[ I(a)=q^{-sK} \int_{\mathfrak o_F}\psi(a\varpi^{Kd}x^d)\abs x^{s-1}\Phi_K(x)\,\mathrm dx \]
so that
\[ \abs{I(a)} \leq q^{-K\sigma } \norm{\Phi}_\infty \int_{\mathfrak o_F} \abs x^{\sigma -1}\,\mathrm dx
\leq q^{-K\sigma } \norm{\Phi}_\infty \zeta_F(\sigma ), \]
hence
\[\abs{I(a)} 
\leq \zeta_F(\sigma ) \norm{\Phi}_\infty \max\left( q^{2(c+1)\sigma/d} \abs a^{-\sigma/d}, q^{(n(\Phi)+c+1)\sigma/(d-1)} \abs a^{-\sigma/(d-1)}\right). \]
As a consequence, when $\abs a\geq 1$, we find
\[ \abs{I(a)} \leq \zeta_F(\sigma )\norm{\Phi}_\infty 
\max( q^{2(c+1)\sigma/d} , q^{(n(\Phi)+c+1)\sigma/(d-1)}) 
\abs{a}^{-\sigma/d}. \]
Since $\abs{I(a)}\leq\zeta_F(\sigma )\norm{\Phi}_\infty$
for any $a\in F$, we conclude that
\begin{equation}
 \abs{I(a)} \leq\zeta_F(\sigma )\norm{\Phi}_\infty
 \max( q^{2(c+1)\sigma/d} , q^{(n(\Phi)+c+1)\sigma/(d-1)}) 
 \max(1,\abs a)^{- \sigma/d}. \end{equation}
\end{proof}


We will need the following 
higher-dimensional generalization of Proposition~\ref{lemm.weyl-s}.
\begin{prop}\label{lemm.weylII-s}
Let $F$ be a local field and $\Phi$ a 
Schwartz function on~$F^n$. Let $n\geq 1$ and let $d_1,\dots,d_n$ be
positive integers. For $s_1,\dots,s_n\in \C$ set 
\[ \kappa(s)=\min\big(1/2,\Re(s_1)/d_1, \dots, \Re(s_n)/d_n \big).\] 
Assume that $\kappa(s)>0$. 
Then, as $\abs a\ra \infty$,
\begin{eqnarray*}
\int_{F^n} \abs{x_1}^{s_1-1} \abs{x_2}^{s_2-1}\dots\abs{x_n}^{s_n-1} 
\psi(a x_1^{d_1}\dots x_n^{d_n}) \Phi(x)\,\mathrm dx \ll 
\\
\min(1,\abs a^{-\kappa(s)})
\prod_{j=1}^n\zeta_F(\Re(s_j)). 
\end{eqnarray*}
This bound is uniform when~$\Phi$ ranges over a bounded subset
of the space of Schwartz functions on~$F^n$ and all $\Re(s_j)\in\R_+^*$ are
bounded from above. 
\end{prop}

\begin{proof}
We may assume that $\Phi$ is real-valued and that its support is 
contained in the unit polydisk.  For $j\in\{1,\dots,n\}$, set $\sigma_j=\Re(s_j)$.
Let $I(a)$ be this integral. If $\abs a\leq 1$, we bound
the integral from above by the integral of its absolute value,
replacing ~$\psi$ by~$1$. Let us assume
that $\abs a\geq 1$.
For $n=1$, the claim follows from Proposition~\ref{lemm.weyl-s}.
By induction, we assume the result is known for
~$<n$ variables. Writing $s=(s_1,s')$, we obtain
\[ I(a)\ll \prod_{j=2}^n \zeta_F(\sigma_j)
           \int_{\abs{x_1}\leq 1} \abs {x_1}^{\sigma_1-1} 
       \max(1,\abs{ax_1^{d_1}}^{-\kappa(s')})\,\mathrm d x_1. \]
We split this integral according to whether or not $\abs{x_1}\leq\abs a^{-1/d_1}$. 
Put $\sigma=\sigma_1$, $d=d_1$, and $\kappa=\kappa(s')$. We have
\[ \int_{\abs {x}\leq \abs a^{-1/d}} \abs x^{\sigma-1}\,\mathrm dx 
 = \abs a^{-\sigma/d} \zeta_F(\sigma).\]
The second integral equals
\[ \abs a^{-\kappa} \int_{\abs a^{-1/d}\leq \abs x\leq 1 } 
               \abs x^{\sigma-d\kappa-1} \,\mathrm dx. \]
If $F=\R$ or~$\C$, we use polar coordinates and obtain,  up to the measure of~$\mathfrak o_F^*$,
\[ \abs a^{-\kappa} \int_{\abs a^{-1/d}}^1  r^{\sigma-d\kappa -1}\,\mathrm dr
= \frac{\abs a^{-\kappa}-\abs a^{-\sigma/d}}{\sigma-d\kappa}
= \frac 1d \abs a^{-c}, \]
for some $c\in (\kappa,\sigma/d)$. In particular,
\[ c\geq \min ( \kappa(s'), \frac{\sigma_1}{d_1})
         = \min\big( \frac{\sigma_1}{d_1},\dots,\frac{\sigma_n}{d_n}\big)
=\kappa(s). \]
When $F$ is non-archimedean, an analogous inequality holds, with
the real integral replaced by a geometric series.
Combining the inequalities, we have
\[ I(a) \ll \abs a^{-\sigma_1/d_1} \prod_{j=1}^n \zeta_F(\sigma_j)   
       + \abs a^{-\kappa(s)}  \prod_{j=1}^n \zeta_F(\sigma_j)
\ll \prod_{j=1}^n \zeta_F(\sigma_j)   \abs a^{-\kappa(s)}, \]
as claimed.
\end{proof}

\subsection{Igusa integrals with rapidly oscillating phase}

\begin{prop}
\label{lemm.irregular}
Let $\Phi\colon F\times \C
\ra\C$ a function such that the functions $s\mapsto \Phi(x,s)$
are holomorphic for any $x\in F$. Assume that
the functions $x\mapsto\Phi(x,s)$
belong to a bounded subset of the space of smooth compactly supported functions
when $\Re(s)$ belongs to a fixed compact subset of~$\R$.

Let $d$ be a positive integer.
For any $a\in F^*$, there exists an holomorphic function $s\mapsto \eta_a(s)$
defined for $\Re(s)>-1$ such that
\[ \eta_a(s) = \int_{F} \abs{x}^{s-1} \psi(a/x^d) \Phi(x,s)\, \mathrm dx \]
for $\Re(s)>0$.
Moreover, when $\Re(s)$ belongs to a compact subset of~$(-1,+\infty)$,
it satisfies a uniform upper-bound of the form
\[ \abs{\eta_a(s)}\ll {\abs a^{-1/d}}. \]
\end{prop}

\begin{proof}
To give an unified proof, we introduce a version of ``Littlewood--Paley''
decomposition. If $F$ is non-archimedean, let $\theta(x)=1$ if $\abs{x}=1$ and $0$ else;
we then have 
\[
\sum_{n\in\Z} \theta(\varpi^n x)=1\quad \text{for all} \quad x\in F^*,
\]
where $\varpi$ is any uniformizing element of~$F$, with $q=\abs\varpi^{-1}$.
If $F$ is archimedean, let $\theta$ be any smooth nonnegative function 
such that:
\begin{itemize}
\item it is supported in the collar $1/2<\abs{x}<2$;
\item it equals $1$ in a neighborhood of the collar $\abs{x}=1$;
\item the sum $\sum_{n\in\Z} \theta(2^n x)$ is a equal to $1$
everywhere (except for \mbox{$x=0$}).
\end{itemize}
Such functions exist; for instance, take any function $\theta_1$ satisfying the
first two assumptions and which is positive on 
the collar $\frac 23<\abs x<\frac32 $.
Let $\Theta_1(x)=\sum_{n\in\Z }\theta_1(2^n x)$;
it is positive everywhere and satisfies $\Theta(2x)=\Theta(x)$.
Let $\theta(x)=\theta_1(x)/\Theta(x)$. It follows that
for $x\neq 0$, 
$$
1=\sum_{n\in\Z} \theta(2^n x).
$$ 
By assumption, if $\abs{x}\leq 1$,
$\theta(2^n x)=0$ for $n\leq -1$, hence 
$\sum_{n\geq 0} \theta(2^nx)$ equals~$1$ on the unit ball
and is also compactly supported.
We let $\varpi=q=2$ in this case.

Now, one has
\[
\eta_a(s)  =  \int_{F} \abs{x}^{s-1} \psi(a/x^d) \Phi(x,s)
  \sum_{n\in\Z}     \theta(\varpi ^nx)\, \mathrm dx
= \sum _{n\in\Z} q^{-ns} \eta_{a,n}(s), \]
where
\[ \eta_{a,n}(s)= 
   \int_{2^{-1}<\abs x<2} \abs{x}^{s-1} \psi(a\varpi^{dn} /x^d)
     \Phi(\varpi^{-n}x,s) \theta(x) \, \mathrm dx.
\]
Since the support of $\Phi(\cdot,s)$ 
is contained in a fixed compact set of~$F$, there exists an integer~$n_0$ such that 
the individual integrals $\eta_{a,n}(s)$ are $0$ for $n\leq -n_0$.
Using the change of variables $x=1/u$, one finds
\begin{align*}
 \eta_{a,n}(s) & = \int_{2^{-1}<\abs x<2} \abs{x}^{s-1} \psi(a\varpi^{dn} /x^d)
     \Phi(\varpi^{-n}x,s) \theta(x) \, \mathrm dx \\
& = \int_{2^{-1}<\abs u<2} \abs{u}^{-s-1} \psi(a\varpi^{nd} u^d)
     \Phi(\varpi^{-n}/u,s)\theta(1/u)\, \mathrm du.
\end{align*}
By Proposition~\ref{lemm.weyl-s}, applied to the Schwartz function
$$
u\mapsto \abs u^{-s-1}\Phi(\varpi^{-n}/u,s)\theta(1/u),
$$
there exists a real number~$c$
such that
\[ \abs{\eta_{a,n}(s)} 
       \leq c (\abs{a\varpi ^{nd}})^{-1/d}  \leq  c q^{-n} \abs{a}^{-1/d}. \]
(In fact, in the ultrametric case, there exists an integer~$n_1$
such that these integrals are zero for $n\geq n_1$.)
It follows that the series defining $\eta_a(s)$ is bounded term by term by
\[ c  \sum_{n\geq -n_0} q^{-n(s+1)} \abs{a}^{-1/d}
  = c \abs{a}^{-1/d} q^{n_0(s+1)} \frac{1}{1-q^{-(s+1)}} . \]
Thus, the series $\eta_a(s)=\sum \eta_{a,n}(s)$ converges 
normally for $\Re(s+1)>0$, locally uniformly in $a\in F^*$
and locally uniformly in $s$, proving the holomorphy in $s$.
The lemma is thus proved.
\end{proof}

\section{Compactifications of additive groups}
\label{sec.additif}

In this chapter, we use methods of harmonic analysis 
to derive asymptotic formulas for the number of 
integral points of bounded height on partial
(equivariant) compactifications of additive groups.

\subsection{Setup and notation}

\subsubsection{Algebraic number theory}
Let $F$ be a number field.
Let $\Val(F)$ be the set of places of~$F$. For $v\in\Val(F)$,
let $F_v$ be the completion of~$F$ at the place~$v$, 
and $\AD_F$ the adele ring of~$F$.
We also fix a finite set~$S\subset\Val(F)$ containing the archimedean places.
We write $\zeta_{F,v}$ for the local factor of Dedekind's zeta function
at a place~$v\in \Val(F)$, and $\zeta^S_F$ for the 
Euler product over places~$v\not\in S$. This product converges
for $\Re(s)>1$ and has a meromorphic continuation to the
whole complex plane, with a single pole at $s=1$; its
residue at $s=1$ is denoted by~$\zeta_F^{S,*}(1)$.
Finally, $\AD_F^S$ is the ring of adeles outside the places belonging to~$S$.

\subsubsection{Algebraic geometry}\label{subsubsec.add-setup1}
Let $G$ be the group scheme~$\ga^n$,
and let $X$ be a smooth projective equivariant compactification of~$G$
over a number field~$F$.
The geometry of such compactifications has been investigated
in~\cite{hassett-t99}. We recall the key facts.
The boundary $X\setminus  G$ decomposes as a union
of $F$-irreducible divisors :
$X\setminus G = \bigcup_{\alpha\in\mathcal A} D_\alpha$,
forming a basis of the group~$\Pic(X)$ of equivalence classes of divisors,
and a basis of the monoid~$\Lambda_\eff(X)$ 
of classes of effective divisors in~$\Pic(X)$. 

Since $\Pic(G)=0$, line bundles on~$X$ have a $G$-linearization,
which is unique up to a scalar since $G$ carries no non-constant invertible functions.
Thus, any  line bundle possesses a meromorphic global section, unique modulo scalars,
whose divisor does not meet~$G$, hence is a linear combination of the~$D_\alpha$.
Therefore, we will freely identify line bundles on~$X$ with divisors contained
in the boundary, and with their classes in the Picard group.

Note that we do not assume 
that the divisors~$D_\alpha$ are geometrically irreducible.
Let $\bmA$ be the set of irreducible components of~$(X\setminus G)_{\bar F}$;
this is a finite set with an action of the Galois group~$\Gamma_F=\Gal(\bar F/F)$
whose set of orbits identifies with~$\mathscr A$. 
More generally, for any extension~$E$ of~$F$ together with
an embedding of~$\bar F$ in~$\bar E$, the set of orbits
of~$\bmA$ under the natural action of~$\Gal(\bar E/E)$
is identified with the set of irreducible components of~$(X\setminus G)_E$.
As above, the classes of these irreducible components form a
basis of the Picard group~$\Pic(X_E)$, as well as a basis of
its effective cone~$\Lambda_\eff(X_E)$.

For each $\alpha\in\mathcal A$, let $F_\alpha$ be the algebraic
closure of $F$ in the function field of $D_\alpha$. It is a finite
extension of $F$. 
After choosing a particular geometrically 
irreducible component of~$D_{\alpha,\bar F}$
(\ie, a specific element in the orbit in~$\bmA$ corresponding to~$\alpha$),
we may view~$F_\alpha$ as a subfield of~$\bar F$; we write $\Gamma_\alpha$
for the Galois group~$\Gal(\bar F/F_\alpha)$.
The representation of the Galois group~$\Gamma_F$ on~$\Pic(X_{\bar F})$ is the direct sum 
of the permutation modules $\Ind_{F_\alpha}^F[\mathbf 1]$ obtained by inducing
the trivial representation from~$\Gal(\bar F/F_\alpha)$ to~$\Gal(\bar F/F)$.

Let $K_X$ be the canonical class of~$X$, \ie, the class of the divisor
of any meromorphic differential form of top degree. In fact,
up to multiplication by a scalar, there is a unique $ G$-invariant
meromorphic differential form~$\omega_X$ on~$X$; its restriction to~$ G$
is proportional  to the form~$\mathrm dx_1\wedge\dots\wedge\mathrm dx_n$.
The anticanonical class~$K_X^{-1}$ is effective; indeed,
writing $\sum_{\alpha\in\mathscr A}\rho_\alpha D_\alpha$ 
for \emph{minus} the divisor of~$\omega_X$,
we have $\rho_\alpha\geq 0$ for any~$\alpha$; in fact,
$\rho_\alpha\geq 2$ (\cite{hassett-t99}, Theorem~2.7).

We also recall that $\mathrm H^i(X,\mathscr O_X)=0$ for $i>0$.
Indeed, $X$ is birational to the projective space~$\P^n$ 
and these cohomology groups
are birational invariants of smooth proper varieties in characteristic zero.

\subsubsection{Adelic metrics and heights}
\label{sect:height}
Endow each line bundle $\mathscr O(D_\alpha)$ with a smooth adelic metric,
as in Section 2.2.3 of~\cite{chambert-loir-tschinkel2009a}. 
The line bundles $\mathscr O(D_\alpha)$ have a canonical
section which we denote by~$\mathsec f_\alpha$. 

For $\mathbf s\in\C^{\mathscr A}\simeq\Pic(X)\otimes\C$
and $\mathbf x =(\mathbf x_v)_v\in G(\AD_F)$,
we let
\[ H(\mathbf x;\mathbf s)
      = \prod_{\alpha\in\mathscr A}\prod_v \norm{\mathsec
f_\alpha}_v(\mathbf x_v)^{-s_\alpha}.
\]
When $\mathbf s$ corresponds to a very ample class~$\lambda$ in $\Pic(X)$,
the restriction of $H(\cdot;\mathbf s)$ to $ G(F)$
is the standard (exponential) height  relative to the projective embedding
of~$X$ defined by~$\lambda$.  In particular, Northcott's theorem
asserts that for any real number~$B$, the set of $x\in G(F)$ such
that $H(\mathbf x;\mathbf s)\leq B$ is finite.
When all components~$s_\alpha$ of~$\mathbf s$ are positive,
the corresponding line bundle~$\lambda$ belongs to the interior
of the effective cone; by Prop.~4.3 of~\cite{chambert-loir-t2002},
this  finiteness property still holds. 

\subsubsection{Partial compactifications}
A partial compactification of~$G$ is a 
smooth quasi-projective scheme~$U$, containing~$G$ as an open subset,
endowed with an action of~$G$ which extends the translation action on~$G$.
We will always assume, as we may, that $U$ is the complement to a reduced divisor~$D$ in a smooth projective
equivariant compactification~$X$ of~$G$ as above. The divisor~$D$ will
be called the boundary divisor of~$U$. 
We let also $\mathscr A_D$ to be the subset of~$\mathscr A$ such that
\[ D=\sum_{\alpha\in\mathscr A_D} D_\alpha. \]

The log-canonical class of~$U$ in~$\Pic(X)$ is the class of~$K_X+D$, the log-anticanonical class is its opposite. 
Since $\rho_\alpha\geq 2$ for all $\alpha\in\mathscr A$,
$-(K_X+D)$ belongs to the interior of the effective cone of~$X$,
so is big.
We have introduced in~\cite{chambert-loir-tschinkel2009a},
Definition~2.2, a virtual $\Gamma_F$-module~$\EP(U)$,
the difference of the Galois modules
$\mathrm H^0(U_{\bar F},\mathscr O^*)/\bar F^*$
and $\Pic(U)/\mathrm{torsion}$ (both abelian groups are free of finite rank).

\begin{lemm}\label{lemm.P(U)-additif}
The virtual representation~$\EP(U)$
is given by
\[ - \sum_{\alpha\in\mathcal A\setminus \mathcal A_D} \Ind_{F_\alpha}^F[\mathbf 1]. \]
where~$[\mathbf 1]$ is the abelian group~$\Z$ together with the trivial action of~$\Gamma_F$,
and $\Ind_{F_\alpha}^F$ denotes the induction functor
to~$\Gamma_F$
from its subgroup of finite index~$\Gamma_{\alpha}$.
\end{lemm}
Note that this representation is trivial if and only if $U=G$.
\begin{proof}
One has $\mathrm H^0(U_{\bar F},\gm)={\bar F}^*$, because $U$ contains~$G$
over which the result is already true.
Moreover, the classes of the divisors $D_\alpha\cap U$,
for $\alpha\not\in\mathcal A_D$, form a basis of $\Pic(U_{\bar F})$;
see Proposition~1.1 of~\cite{chambert-loir-t2002} when $D=\emptyset$,
but the proof holds for any equivariant embedding of~$ G$.
\end{proof}

\subsubsection{Clemens complexes}
We have introduced in~\cite{chambert-loir-tschinkel2009a} various simplicial complexes to encode
the combinatorial properties of the irreducible components
of the boundary divisors~$D_\alpha$ and their intersections.
Let $(V,Z)$ be a pair consisting of a smooth variety over a field~$F$
and of a divisor~$Z$ such that $Z_{\bar F}$ has strict normal crossings,
\ie, is a sum of smooth irreducible components which meet transversally.

Let $\bmA$ be the set of irreducible components of~$Z_{\bar F}$,
together with its action of~$\Gamma_F$; 
for $\alpha\in\bmA$,
let $Z_\alpha$ be the corresponding component.
For any subset~$A$ of~$\bmA$, let 
\[ Z_A=\bigcap_{\alpha\in\mathscr A}Z_\alpha, \quad
Z_A^\circ=Z_A\setminus \left( \bigcup_{\beta\not\in A} Z_\beta\right). \]
The sets $Z_A$ are closed in~$V_{\bar F}$,
the sets~$Z_A^\circ$ form a partition of~$V$ in locally closed subsets.
In particular, $Z_{\emptyset}^\circ=(V\setminus Z)_{\bar F}$. 
Unless they are empty, the sets~$Z_A$ and~$Z_A^\circ$
 are defined over~$F$ if and only if~$A$
is globally invariant under~$\Gamma_F$.

The geometric Clemens complex $\Cl_{\bar F}(Z)$ of the pair~$(V,Z)$
has for vertices the elements of~$\bmA$; more generally, its
faces are the irreducible components of the closed subsets~$Z_A$, 
for all non-empty subsets $A\subset\bmA$.
An $n$-dimensional face corresponds to a component of codimension $n+1$ in~$V$. 

The geometric Clemens complex carries a 
natural simplicial action of the group~$\Gamma_F$.
The rational Clemens complex, $\Cl_F(Z)$, of the pair~$(V,Z)$
has for faces the $\Gamma_F$-invariant faces of~$\Cl_{\bar F}(Z)$.
A similar complex~$\Cl_E(Z)$ can be defined for any extension~$E$ of~$F$;
we will apply this when $E=F_v$ is the completion of~$F$ at a place~$v$
of~$F$.

For any place~$v$ of~$F$, the $v$-analytic Clemens complex $\Clan_{F_v}(Z)$
is then defined as the subcomplex of~$\Cl_{F_v}(Z)$ 
whose faces correspond to irreducible components of 
intersections of irreducible components of~$Z_{F_v}$, 
which contain $F_v$-rational points.

\medskip

We will apply these considerations when $V=X$ and $Z=D=X\setminus U$,
where $X$ is a smooth projective equivariant compactification of~$G$, 
and $U\subset X$ is a partial compactification. 
We require throughout this paper that over~$\bar F$ 
the divisor~$X\setminus G$ has strict normal crossings.

\subsubsection{Measures}
Let us fix a gauge form~$\mathrm d\mathbf x$ on~$G$, defined over~$F$.
For any place~$v$ of~$F$, its absolute value is a Haar measure
on~$G(F_v)$, still denoted~$\mathrm d\mathbf x$ (or~$\mathrm d\mathbf x_v$).
The product of these Haar measures is a Haar measure on~$G(\AD_F)$.
Moreover, $G(F)$ is a discrete cocompact subgroup of covolume~$1$
(see~\cite{tate67} for the case $G=\ga$; the general case follows
from it).

The fixed adelic metrization of the line bundles~$\mathscr O(D_\alpha)$
induces an adelic metrization on the canonical and log-canonical line bundles.
As in Section~4 of~\cite{chambert-loir-tschinkel2009a},
this gives rise, for any place~$v$ of~$F$, to measures
$\tau_{X,v}$ on the $F_v$-analytic manifold~$X(F_v)$ and 
its restriction $\tau_{U,v}$ to the open submanifold $U(F_v)$ of~$X(F_v)$.
We also introduced
the measure 
\[ \tau_{(X,D),v}=\norm{\mathsec f_D}_v^{-1}\tau_{U,v} \]
on $U(F_v)$, where $\mathsec f_D$ is the canonical section
of~$\mathscr O_X(D)$. 

By  results in Section~4 of~\cite{chambert-loir-tschinkel2009a}, 
the product of the local measures 
$$
\mathrm L_v(1,\EP(U)) \tau_{U,v}, \quad \quad \text{ for } \quad v\not\in S,
$$
converges to a measure on the adelic space $U(\AD_F^S)$ outside~$S$.
We then define
\[ \tau_U^S = \mathrm L_*^S(1,\EP(U))^{-1} \prod_{v\not\in S} \mathrm L_v(1,\EP(U)) \tau_{U,v}. \]

Let $v$ be a place in~$S$ 
and fix a decomposition group $\Gamma_v\subset\Gamma_F$ at~$v$.

Let $\mathscr A_v=\bmA/\Gamma_v$  and $\mathscr A_{D,v}$ be the set of
orbits of the group~$\Gamma_v$ acting on~$\bmA$ and $\bmA_D$. 
They are equal to the sets of $F_v$-irreducible components of~$(X\setminus G)_{F_v}$ and of $D_{F_v}$, respectively.
For each~$\alpha\in\mathscr A$,
the divisor~$(D_\alpha)_{F_v}$ decomposes as a sum of irreducible components
$D_{\alpha,w}$ indexed by the direct factors~$F_{\alpha,w}$
of the algebra $F_\alpha\otimes_F F_v$. The canonical section~$\mathsec f_\alpha$
of~$\mathscr O(D_\alpha)$ decomposes accordingly as a product
$\prod \mathsec f_{(\alpha,w)}$.
We shall identify $\mathscr A_v$ with the set of such pairs $(\alpha,w)$,
and $\mathscr A_{D,v}$ with those such that $\alpha\in\mathscr A_D$.

Let $A$ be a face of the analytic Clemens complex $\Clan_{F_v}(X,D)$,
that is a subset of~$\mathscr A_{D,v}$ such that
the intersection~$D_A$ of the divisors $D_{\alpha,w}$, for $(\alpha,w)\in A$, 
has  a common $F_v$-rational point. 
Iterating the construction given in Section~2.1.13
of~\cite{chambert-loir-tschinkel2009a},
we have ``residue measures'' 
on the submanifold~$D_A(F_v)$ of~$X(F_v)$, 
normalized by incorporating the product $\prod_{\alpha\in A} \mathrm c_{F_{\alpha},w}$
associated to the local fields~$F_{\alpha,w}$, for $(\alpha,w)\in A$.
The resulting measures are denoted by~$\tau_{D_A,v}$.

\subsubsection{Integral points}
\label{sect:integral-points}
Let $S$ be a finite set of places of~$F$ containing the archimedean places.
We are interested in ``$S$-integral points of $U$''.
The precise definition depends on the choice of a model of $U$ over the
ring~$\mathfrak o_{F,S}$ of $S$-integers in~$F$; namely a quasi-projective
scheme over~$\Spec(\mathfrak o_{F,S})$ whose restriction to the generic fiber
is identified with~$U$. 
Given such a model,
the $S$-integral points of~$U$ are the elements of~$\mathscr U(\mathfrak o_{F,S})$,
in other words those rational points of~$U(F)$ which
extend to a section of the structure morphism 
from~$\mathscr U$ to~$\Spec(\mathfrak o_{F,S})$.
Fix such a model~$\mathscr U$.

For any finite place~$v$ of~$F$ such that $v\not\in S$,
let $\mathfrak u_v=\mathscr U(\mathfrak o_{F,v})$ and let $\delta_v$ be the characteristic function
of the subset $\mathfrak u_v\subset X(F_v)$.
Consequently, 
\emph{a point $x\in X(F)$ is an $S$-integral point of~$U$
if and only if one has $x\in \mathfrak u_v$ for any place~$v$ of~$F$ such that $v\not\in S$},
equivalently, if and only if $\prod_{v\not\in S}\delta_v(x)=1$.
We put $\delta_v\equiv 1$ when $v\in S$.

By the definition of an adelic metric, there exists a finite set of places~$T$,
a flat projective model $\mathscr X$ over $\mathfrak o_{F,T}$
satisfying the following properties:
\begin{itemize}
\item
$\mathscr X$
is a smooth equivariant compactification of
the $\mathfrak o_{F,T}$-group scheme~$G$;
\item
for any~$\alpha\in\mathscr A$, the closure~$\mathscr D_\alpha$ of $D_\alpha$ 
is a divisor on $\mathscr X$;
\item 
the boundary $\mathscr X\setminus G$ is the union
of these divisors $\mathscr D_\alpha$;
\item
for any~$\alpha$,
the section $\mathsec f_\alpha$ extends to a global section
of the line bundle $\mathscr O(\mathscr D_\alpha)$ on $\mathscr X$
whose divisor is precisely $\mathscr D_\alpha$.
\end{itemize}
Since
any isomorphism between two $\mathfrak o_{F,S}$-schemes of finite presentation
extends uniquely to an isomorphism over  an open subset of~$\Spec\mathfrak o_{F,S}$,
we may assume that after restriction to~$\Spec(\mathfrak o_{F,T})$, $\mathscr U$ is the complement
in~$\mathscr X$ to the Zariski closure~$\mathscr D=\sum_{\alpha\in \mathscr A_D}\mathscr D_\alpha$
of~$D$ in~$\mathscr X$.
For all places~$v\not\in T\cup S$,
one thus has $\mathfrak u_v=\mathscr U(\mathfrak o_v)$.

\subsubsection{Height zeta function}
We proceed to study of the distribution of $S$-integral
points of~$U$ with respect to heights. In fact, we only study those
integral points which belong to $ G(F)$; moreover, we 
consider the heights with respect to all line bundles in the Picard group
at the same time.
The \emph{height zeta function} is
defined as
\[ \Zeta (\mathbf s) = \sum_{\mathbf x\in G(F)\cap\mathscr U(\mathfrak o_{F,S})} 
	    H(\mathbf x;\mathbf s)^{-1},
\]
whenever it converges.
It follows from Proposition~4.5 in~\cite{chambert-loir-t2002}
that there exists a non-empty open subset $\Omega\subset\Pic(X)_\R$
such that $\Zeta(\mathbf s)$ converges absolutely to a bounded
holomorphic function in the tube domain $\Tube(\Omega)=\Omega+\mathrm i\Pic(X)_\R$.

Using the functions~$\delta_v$ 
defined above, it follows that
\[ \Zeta(\mathbf s) = \sum_{\mathbf x\in G(F)}
 \prod_{v\in\Val(F) \setminus S} \big(  \delta_v(\mathbf x_v)  \prod_{\alpha\in\mathscr A} 
\norm{\mathsf f_\alpha(\mathbf x_v)}_v^{s_\alpha}\big) 
\times
\prod_{v\in S} \big(\prod_{\alpha\in\mathscr A}
\norm{\mathsf f_\alpha(\mathbf x_v)}_v^{s_\alpha}\big).
\]

\subsubsection{Fourier transforms}
We write $\langle\cdot,\cdot\rangle$ for the bilinear pairing on~$G$ 
\[ 
\langle (x_1,\dots,x_n),(y_1,\dots,y_n)\rangle=\sum x_i y_i. 
\]
Then, the maps
\[  G(F_v)\times  G(F_v)\ra\C^* , \quad
  ( \mathbf a, \mathbf x)\mapsto \psi_v(\langle \mathbf a,\mathbf x\rangle)
\]
and
\[ G(\AD_F)\times G(\AD_F)\ra\C^*, \quad (\mathbf a,\mathbf x)\mapsto \psi(\langle \mathbf a,\mathbf x\rangle) \]
are perfect pairings.

For each place~$v$ of $F$ and $\mathbf a\in G(F_v)$, define
\[ \hat H_v( \mathbf a;\mathbf s)
     = \begin{cases}\displaystyle
\int_{ G(F_v)} \delta_v(\mathbf x)
	 \prod_{\alpha\in\mathscr A}   \norm{\mathsec f_\alpha(\mathbf x)}_v^{s_\alpha}
	   \psi_v(\langle  \mathbf a,\mathbf x\rangle) \, \mathrm d\mathbf x
& \text{if $v\not\in S$} \\
\displaystyle \int_{ G(F_v)} 
	\prod_{\alpha\in\mathscr A}    \norm{\mathsec f_\alpha(\mathbf x)}_v^{s_\alpha}
	   \psi_v(\langle  \mathbf a,\mathbf x\rangle) \, \mathrm d\mathbf x
& \text{otherwise}, \end{cases}
\]
where the integrals are taken with respect to the chosen 
Haar measure on~$G(F_v)$.
Moreover, for any $\mathbf a=(\mathbf a_v)_v\in G(\AD_F)$, we define
\[ \hat H(\mathbf a;\mathbf s)
 = \prod_{v\in\Val(F)} \hat H_v(\mathbf a_v;\mathbf s). \]

Formally, we can write down the Poisson summation formula  
for the locally compact group~$G(\AD_F)$
and its discrete cocompact subgroup~$G(F)$. Since $G(F)$ has covolume~$1$,
\begin{equation}
\label{eqn:poisson} 
\Zeta(\mathbf s) = \sum_{ \mathbf a\in G(F)}
      \hat H(\mathbf a;\mathbf s). 
\end{equation}
Below, we will establish a series of analytic
estimates guaranteeing that the Poisson summation formula can be applied
and we will show that its right hand side provides a meromorphic continuation
of the height zeta function.

\subsection{Fourier transforms (trivial character)}

In our paper~\cite{chambert-loir-tschinkel2009a}, we have established
the analytic properties of the Fourier transform at the trivial character,
\ie, of the local integrals
\[ \hat H_v(0;\mathbf s) = \int_{G(F_v)} \delta_v(\mathbf x) \prod_{\alpha \in\mathscr A} \norm{\mathsec f_\alpha(\mathbf x)}_v^{s_\alpha}\,\mathrm d\mathbf x \]
for $v\in\Val(F)$, and of their ``Euler product''
\[ \hat H(0;\mathbf s) = \prod_{v\in\Val(F)} \hat H_v(0;\mathbf s). \]
We now summarize these results.

If $a$ and~$b$ are real numbers, we define
$\Tube_{>a}$, \resp  $\Tube_{(a,b)}$, as the set of $s\in\C$
such that $a<\Re(s)$, \resp $a < \Re(s) <b$. We write
$\Tube_{>a}^{\mathscr A}$ for the set of families of elements
of~$\Tube_{>a}$ indexed by the set~$\mathscr A$, etc.

\subsubsection{Absolute convergence of the local integrals}

We begin by stating the domain of absolute convergence of the 
local integrals~$\hat H_v(0;\mathbf s)$, as well as
their meromorphic continuation.
We recall that $\rho=(\rho_\alpha)_{\alpha\in\mathscr A}$ is the vector of positive
integers such that $\sum_{\alpha\in\mathscr A}\rho_\alpha D_\alpha$ is an anticanonical
divisor.
\begin{lemm}\label{lemm.local-meromorphic}
Let $v$ be a place of~$F$.
The integral defining
$$
\hat H_v(0;(s_\alpha+\rho_\alpha-1)_{\alpha\in \mathscr A})
$$
converges for $s\in\Tube_{>0}^{\mathscr A}$ to
a holomorphic function on $\Tube_{>0}^{\mathscr A}$.
The holomorphic function 
\[ \mathbf s\mapsto \prod_{(\alpha,w)\in\mathscr A_v} \zeta_{F_\alpha,v}(s_\alpha)^{-1}
\hat H_v(0;(s_\alpha+\rho_\alpha-1)) \]
extends to a holomorphic function on~$\Tube_{>-1/2}^{\mathscr A}$.
\end{lemm}
\begin{proof}
With the notation of
Lemma~4.1 in~\cite{chambert-loir-tschinkel2009a},
we have 
\[ \mathrm d\mathbf x=\mathrm d\tau_{(X,D),v} = 
\prod_{\alpha\in\mathscr A} \norm{\mathsec f_\alpha(\mathbf x)}_v^{-\rho_\alpha}
\,\mathrm d\tau_X. \]
Consequently, the convergence assertion,
follows from this Lemma  
by taking for~$\Phi$ the function~$\delta_v$ introduced above.
The meromorphic continuation is then a particular case
of Proposition~4.2 of that paper.
\end{proof}

\subsubsection{Places in~$S$}

We fix a place $v\in \Val(F)$ such that $v\in S$.
For $\lambda\in\R_{>0}^{\mathscr A}$, let 
\[ a(\lambda,\rho) = \max_{\alpha\in\mathscr A} \frac{\rho_\alpha-1}{\lambda_\alpha} \]
and let $\mathscr A(\lambda,\rho)$ be the set of all $\alpha\in\mathscr A$
where the maximum is achieved. By Lemma~\ref{lemm.local-meromorphic},
the integral defining $\hat H_v(0;s\lambda)$ 
converges for any complex number~$s$ such that $\Re(s)>a(\lambda,\rho)$,
defines a holomorphic function of~$s$ 
in the tube domain~$\Tube_{>a(\lambda,\rho)}$,
and has a meromorphic continuation to a tube~$\Tube_{>a(\lambda,\rho)-\delta}$,
for some positive real number~$\delta>1/\max(\lambda_\alpha)$, 
with a pole of order at most~$\#\mathscr A_v$ at $s=a(\lambda,\rho)$.

In order to state a precise answer, let
us introduce the simplicial complex
$\Clan_{F_v,(\lambda,\rho)}(X\setminus G)$ 
obtained from~$\Clan_{F_v}(X\setminus G)$ by removing
all faces containing a vertex~$(\alpha,w)\in\mathscr A_v$ such
that $\rho_\alpha-1< a(\lambda,\rho) \lambda_\alpha$.

\begin{prop}
There exist a positive real number~$\delta$, and for each face~$A$
of~$\Clan_{F_v,(\lambda,\rho)}(X\setminus G)$ 
of maximal dimension, a holomorphic
function~$\phi_{A}$ defined on~$\Tube_{>a(\lambda,\rho)-\delta}$
with polynomial growth in vertical strips
such that
\[ \phi_{A}(a(\lambda,\rho)) 
 =  \int_{D_A(F_v)} \prod_{(\alpha,w)\not\in A} \norm{\mathsec f_{(\alpha,w)}(\mathbf x)}_v^{a(\lambda,\rho)\lambda_\alpha-1} \mathrm d\tau_{D_A,v}(\mathbf x)
 \]
and such that for any $s\in\Tube_{>a(\lambda,\rho)}$, one has
\[
  \hat H_v(0;s\lambda  ) 
 = \sum_A  \phi_{A}(s) \prod_{(\alpha,w)\in A} \zeta_{F_{\alpha,w}}(\lambda_\alpha (s-a(\lambda,\rho)) ) , \]
where the sum ranges over the faces~$A$ of~$\Clan_{F_v,(\lambda,\rho)}(D)$
of maximal dimension.
In particular, the order of the pole of~$\hat H_v(0;\lambda s)$ at $s=a(\lambda,\rho)$ is given by
\[ b_v(\lambda,\rho) = 1+\dim \Clan_{F_v,(\lambda,\rho)}(D). \]
\end{prop}
\begin{proof}
This is a special case of~\cite{chambert-loir-tschinkel2009a},
Proposition~4.3.
\end{proof}

 \subsubsection{Places outside~$S$; Denef's formula}

When $v\not\in S$, we first refine
the statement of Lemma~\ref{lemm.local-meromorphic}, by
taking into account
the compactly supported function~$\delta_v$ on~$U(F_v)$.
The same proof shows that the inequalities $\Re(s_\alpha)>\rho_\alpha-1$, 
for $\alpha\in\mathscr A_D$, ensure the
absolute convergence of the integral $\hat H_v(0;\mathbf s)$.

Moreover, for places of good reduction, \ie, places~$v\not\in T\cup S$, 
we may apply Denef's formula 
(Proposition~4.5 of~\cite{chambert-loir-tschinkel2009a}) and obtain
the explicit formula
\begin{equation}
\hat H_v(0;(s_\alpha)) = (q_v^{-1}\mu_v(\mathfrak o_v))^{\dim X}
 \sum_{A\subset (\bmA \setminus\bmA_D)/\Gamma_{v}}
  \# \mathscr D_A^0(k_v) \prod_{\alpha\in A} \frac{q_v^{f_\alpha}-1}{q_v^{f_\alpha (s_\alpha-\rho_\alpha)}-1}. \end{equation}
(An $\alpha\in A$ corresponds to an $F_v$-irreducible component
of~$(X\setminus G)_{F_v}$ which is not contained in~$D_{F_v}$.
By the good reduction hypothesis, this component is split over an unramified extension of~$F_v$, of degree~$f_\alpha$.)

 \subsubsection{Euler product}
We have explained in~\cite{chambert-loir-tschinkel2009a},
Section~4.3.3, how to derive
from this explicit formula the analytic behavior
of the infinite product of~$\hat H_v(0,\mathbf s)$ over all places~$v\not\in S$.
Following the proof of 
Proposition~4.10 in that paper,
we obtain:
\begin{prop}
The infinite product 
\[ \hat H^S(0;\mathbf s)= \prod_{v\not\in S} \hat H_v(0;\mathbf s) \]
converges absolutely for any $\mathbf s\in\C^{\mathscr A}$ such that
$\Re(s_\alpha)>\rho_\alpha+1$ for all $\alpha\in\mathscr A\setminus\mathscr A_D$.
Moreover, there exists a holomorphic function~$\phi^S$ 
defined on the set of~$\mathbf s\in\C^{\mathscr A}$ 
such that $\Re(s_\alpha)>\rho_\alpha+1/2$ for $\alpha\in\mathscr A\setminus\mathscr A_D$ 
which has polynomial growth in vertical strips and satisfies
\[ \hat H^S(0;\mathbf s) = \phi^S(\mathbf s) \prod_{\alpha\not\in\mathscr A_D} \zeta_{F_\alpha}^S(s_\alpha-\rho_\alpha). \]
Moreover, for any point~$\mathbf s\in\C^{\mathscr A}$ such that
$s_\alpha=\rho_\alpha+1$ for all $\alpha\not\in\mathscr A_D$, one has
\[ \phi^S(\mathbf s) = \prod_{\alpha\not\in\mathscr A_D} \zeta^{S,*}_{F_\alpha}(1) \int_{U(\AD_F^S)} \prod_{v\not\in S}\delta_v(\mathbf x)\,
 \prod_{\alpha\not\in\mathscr A_D} \prod_{v\not\in S}
      \norm{\mathsec f_\alpha(\mathbf x)}_v^{s_\alpha-1}\,
  \mathrm d\tau_U(\mathbf x). \]
\end{prop}

\subsection{The Fourier transforms at non-trivial characters}

  \subsubsection{Vanishing of the Fourier transforms outside of a lattice}

We proceed with an elementary, but important, remark --- 
we assume that the written integrals converge absolutely.
\begin{lemm}
For each finite place~$v$, there exists a compact open subgroup $\mathfrak d_{X,v}\subset G(F_v)$
such that $\hat H_v(\mathbf a;\mathbf s)=0$ for $\mathbf a\not\in\mathfrak d_{X,v}$.
Moreover, $\mathfrak d_{X,v}=G(\mathfrak o_v)$ for almost all finite places~$v$.

There exists a lattice~$\mathfrak d_X\subset G(F)$ such that $\hat H(\mathbf a;\mathbf s)=0$
for $\mathbf a\not\in \mathfrak d_X$.
\end{lemm}
\begin{proof}
For each finite place~$v$, the function $\mathbf x\mapsto H_v(\mathbf x;\mathbf s)$ 
on~$G(F_v)$  is invariant under the action of an open subgroup of~$G(F_v)$,
which equals $G(\mathfrak o_v)$ for almost all~$v$. (See~\cite{chambert-loir-t2002}, Proposition~4.2.)
Consequently, the Fourier transform vanishes at any character whose restriction
to this open subgroup is non-trivial. This establishes the claim.
\end{proof}

\subsubsection{Archimedean places: integration by parts}\label{sec:non-trivial-varch}
In order to be able to prove the convergence of the right
hand side of the Poisson formula (Eq.~\eqref{eqn:poisson})
we need to improve the decay at infinity of the Fourier transforms
at archimedean places of~$F$. 

Let $v$ be such a place.
As in~\cite{chambert-loir-t2002}, we use integration by parts
with respect to vector fields on~$X$ which extend the
invariant vector fields on~$G$.
According to Prop.~2.1 in~\loccit,  any invariant
vector field on~$G$ extends uniquely 
to a regular vector field~$\partial^X$ on~$X$;
 moreover,  Prop.~2.2 there asserts that 
for any local equation~$z_\alpha$ of
a boundary component~$D_\alpha$, $z_\alpha^{-1}\partial^X z_\alpha$
is regular along~$D_\alpha$.
As in Prop.~8.4 of this paper,
we can perform repeated integration
by parts and write, for $\mathbf s$ in the domain of absolute convergence:
\begin{equation}
\label{eq:fourier-varch}
 \hat H_v(\mathbf a;\mathbf s)= (1+\norm {\mathbf a}_v)^{-N}\int_{G(F_v)}
   H_v(\mathbf x;\mathbf s)^{-1} \psi_v(\langle\mathbf a;\mathbf x\rangle)
   h_N(\mathbf a;\mathbf x;\mathbf s)\,\mathrm d\mathbf x\end{equation}
where $h_N$ is a smooth function on~$X(F_v)$ which admits
a uniform upper-bound of the form
\[ \abs{h_N(\mathbf a;\mathbf x;\mathbf s)}
 \ll (1+\norm {\mathbf s})^N .\]

The important observation to make is that 
these integration by parts preserve the form of the Fourier integrals
when they are written in local coordinates.
Consequently, we can apply the techniques
developed in Section~\ref{sec.non-trivial-S}
to the integral expression~\eqref{eq:fourier-varch}
of the Fourier transform.
This shows that  the upper-bounds for the meromorphic continuation
established in that Section can be improved
by a factor 
$$
(1+\norm {\mathbf s})^N/(1+\norm {\mathbf a}_v)^N,
$$ 
where $N$ is an arbitrary
positive integer.

 \subsubsection{Places outside~$S$}
Let $f_{\mathbf a}$ be the rational function on $X$ corresponding to the 
linear form $\langle \mathbf a,\cdot\rangle$ on~$G$. Its divisor takes the form
\[ \div(f_{\mathbf a})= E(f_{\mathbf a})-\sum_{\alpha\in \mathscr A}d_\alpha(\mathbf a) D_{\alpha},  \]
with $d_\alpha(\mathbf a)\ge 0$ for all $\alpha \in \mathscr A$,
and  $E(f_{\mathbf a})$ is the Zariski closure in~$X$ of the hyperplane
with equation $\langle\mathbf a,\cdot\rangle=0$ in~$ G$.
(See~\cite{chambert-loir-t2002}, Lemma~1.4.)
For any~$\mathbf a$, let $\mathscr A_0^D(\mathbf a)$ be
the set of~$\alpha\in\mathscr A\setminus\mathscr A_D$ such
that $d_\alpha(\mathbf a)=0$. 
We also define $T(\mathbf a)$
to be the union of~$S$, $T$ and of the set of finite places~$v\not\in S\cup T$
such that $\mathbf a$ reduces to~$0$ modulo~$v$.

\begin{prop}\label{prop.estimatewithcharacters}
There exists a constant $C(\eps)$ independent of
$\mathbf a\in \mathfrak d_X$
such that for any $v\not\in T(\mathbf a)$
and any~$\mathbf s\in\C^{\mathscr A}$ 
such that $\Re(s_\alpha)>\rho_\alpha-\frac12+\eps$ for any~$\alpha$,
\[  \abs{ 1 - \hat H_v( {\mathbf a} ; \mathbf s )
 \prod_{\alpha\in \mathcal A_0^D(\mathbf a)} 
(1-q_v^{-f_\alpha (1+s_\alpha-\rho_\alpha)})  }
     \leq C(\eps) q_v^{-1-\eps}. \]
\end{prop}
\begin{proof}
When $X=U$, \ie, when $U$ is projective and $D=\emptyset$,
this has been proved in~\cite{chambert-loir-t2002},
see Proposition~10.2 and \S11 there. 
A straightforward adaptation of that proof establishes
the general case: by the definition of~$\delta_v$, 
we split the integral as a sum of integrals over the 
residue classes in~$\mathscr U(k_v)$ and each of these integrals
is computed in~\cite{chambert-loir-t2002},
leading to the asserted formula.
\end{proof}

As a consequence, we obtain the following meromorphic continuation
of the infinite product over places~$v\not\in S$ of $\hat H_v(\mathbf a;\mathbf s)$.
\begin{coro}\label{coro.otherchars}
For any $\eps >0$ and $\mathbf a\in \mathfrak d_X \setminus \{ 0\}$
there exists a holomorphic bounded function $\phi(\mathbf a;\cdot)$
on $\Tube_{>-1/2+\eps}^{\mathscr A}$ such
that for any $\mathbf s\in \Tube_{>0}^{\mathscr A}$,
\[ \hat H^S({\mathbf a};\mathbf s+\rho)=\prod_{v\not\in S}
 \hat H_v( {\mathbf a} ; \mathbf s+\rho ) = 
\phi({\mathbf a} ; \mathbf s )\prod_{\alpha\in \mathcal A_0^D(\mathbf a)}
     \zeta_{F_\alpha}(1+s_\alpha). \]
Moreover, there exist a real number~$C(\eps)$ such that
one has the uniform estimate
\[ \abs{\phi({\mathbf a};\mathbf s)} \leq C(\eps)
     (1+\norm{\mathbf a}_\infty)^{\eps}. \]
\end{coro}

\subsection{The Fourier transforms at non-trivial characters (places in~$S$)}
\label{sec.non-trivial-S}
 
 
Here we study the Fourier transforms for a place~$v\in S$,
when the character~$\mathbf a$ is non-trivial.  We will prove
uniform estimates in~$\mathbf a$. When no confusion
can arise, we will remove the index~$v$ from the notation.
Written in local charts of the compactification, our 
integrals take the form
\[ \int_{\mathscr U_A} \prod_{\alpha\in A} \abs{x_\alpha}^{\lambda_\alpha s-\rho_\alpha} \psi_v(f_{\mathbf a}(\mathbf x))  \theta_A(\mathbf x;\mathbf s)\,
\prod_{\alpha\in A}\mathrm dx_\alpha\, \mathrm d\mathbf y, \]
where $\mathbf x=((x_\alpha)_{\alpha\in A},\mathbf y)$ is a system of local coordinates
for~$\mathbf x$ in a neighborhood of a point of $D_A^\circ(F_v)$,
and $\theta_A$ a smooth function with compact support in~$\mathbf x$,
which is holomorphic and has polynomial growth in vertical strips
with respect to the parameter~$\mathbf s$.
See Section~3.3 of~\cite{chambert-loir-tschinkel2009a} for more details.
Note that the stratum~$A$ of the $F_v$-analytic Clemens complex of~$X\setminus G$
consists of elements $\tilde\alpha=(\alpha,w)\in\mathscr A_v$,
the local coordinate~$x_{\tilde\alpha}$ is an element of the local field~$F_{\alpha,w}$; these coordinates are completed by the family~$\mathbf y=(y_\beta)$.

In the following exposition, we assume throughout that all
irreducible components of~$X\setminus G$ are geometrically irreducible.
Then, the local fields $F_{\alpha,w}$ all become~$F_v$, etc.,
simplifying the notation.
(See~\cite{chambert-loir-t2002} for an example of the required adaptation.)

The analytic properties of this integral are determined by
the  divisor of the rational function~$f_{\mathbf a}$. 
The analysis is simpler if this divisor has strict normal crossings;
we can reduce to this case using embedded resolution of singularities.
To obtain uniform estimates, we do this in families.

 The functions $\mathbf a\mapsto d_\alpha(\mathbf a)$ on~$G\setminus\{0\}$
 are constructible,
 upper semi-continuous (they do not increase under specialization),
 and descend to the projective space~$\mathbf P^{n-1}$ of
 lines in~$ G$.

\begin{lemm}[Application of resolution of singularities]
There exists a decomposition~$\mathbf P^{n-1}=\coprod_i P_i$ of~$\mathbf P^{n-1}$ in locally closed
subsets, and, for each $i$,
a map $Y_i\ra P_i\times X$ which is a composition of blowups
whose  centers lie over a nowhere dense subset 
of~$P_i\times (U\setminus G)$ and
are smooth over~$P_i$, such that 
the divisor of the rational functions~$f_{\mathbf a}$ on~$Y_i$
is a relative divisor with normal crossings on~$P_i$.
\end{lemm}

\begin{proof}
First apply Hironaka's theorem 
and perform an embedded resolution of singularities of the 
pair~$(X,\abs{\div(f_{\mathbf a})})$
over the function field of~$\mathbf P^{n-1}$.  Then 
spread out this composition of blow-ups with smooth centers 
by considering the composition
of the blow-ups of their Zariski closures in~$\mathbf P^{n-1}\times X$.
Let $P_1$ be the largest open subset of~$\mathbf P^{n-1}$ over which
all of these centers are smooth; it is dense.
We also observe that the map $Y_1\ra P_1\times X$ 
is an isomorphism over~$G$, 
as well as over the generic points of~$X\setminus G$, because
the divisor of~$f_{\mathbf a}$ is smooth there. 

We repeat this procedure for each irreducible
component of~$\mathbf P^{n-1}\setminus P_1$. 
This process ends by noetherian induction.
\end{proof}

Let us return to our integrals. We fix one of the strata, say~$P$,
and derive uniform estimates for $[\mathbf a]$ in~$P$.
Note that all coefficients~$d_\alpha(\mathbf a)$ are constant on~$P$;
they  will be denoted~$d_\alpha$. We write $\pi\colon Y\ra P\times X$
for the resolution of singularities introduced in the Lemma.
Considering the normalization~$\bar P$ of the closure of~$P$ in~$\mathbf P^{n-1}$,
and the successive blow-ups along the Zariski closure of the
successive centers (which may not be smooth on~$\bar P$ anymore), we extend it 
to a map $\bar\pi\colon\bar Y\ra\bar P\times X$.
We change variables and write our integral as an integral on
the fiber~$\bar Y_{[\mathbf a]}$ over the point~$[\mathbf a]\in P$.
This is relatively innocuous as long as $[\mathbf a]$ lies in a compact
subset of~$P(F_v)$; when $[\mathbf a]$ approaches the
boundary $\partial P=\bar P\setminus P$, 
the upper-bounds we obtain increase, but at most polynomially in the distance to
the boundary.

After resolution of singularities, we can write
the divisor $\div(f_{\mathbf a})$ as 
\[ - \sum_{\alpha\in A}d_\alpha D_\alpha + \sum_{\beta\in B} e_\beta E_{\beta} ,\]
where  
the divisors~$D_\alpha$ are the strict transforms of the one with
the same name, and the divisors~$E_{\beta}$ are either the closure in~$Y$ 
of the hyperplane~$G\cap \div(f_{\mathbf a})$ of~$G$ (for $\beta=0\in B$),
or exceptional divisors of the resolution of singularities (for all other values of~$\beta$). 
The integers~$e_\beta$ are unknown a priori, but
do not depend on~$[\mathbf a]\in P$.
Moreover, the divisor of the Jacobian  of~$\pi$ has the form
\[ \sum_{\beta\in B} \iota_\beta E_{\beta}, \]
where $\iota_\beta>0$ unless $\beta=0$.
For each $\alpha\in A$, let $d_{\alpha,\beta}$ be nonnegative integers such that
\[ \pi^* D_\alpha = D_\alpha+ \sum_{\beta\in B} d_{\alpha,\beta} E_\beta. \]
For each $\beta\in B$ we set
\[ \lambda_\beta(s) = \sum_{\alpha\in A} d_{\alpha,\beta} (\lambda_\alpha s-\rho_\alpha)
                  + \iota_{\beta}. \]
We observe that $\lambda_0\equiv 0$, but that $\lambda_\beta$
is positive at $\sigma=\max(\rho_\alpha/\lambda_\alpha)$, otherwise.

Let us consider local coordinates 
$((x_\alpha)_{\alpha\in A},(y_\beta)_{\beta\in B},(z_\gamma)_{\gamma\in C})$
 around~$\mathbf x$, where for each~$\alpha$,
$x_\alpha$ is a local equation of~$D_\alpha$,
for each~$\beta\in B$, $y_\beta$ is a local equation of~$E_\beta$. 
By construction, the function 
\[ u_{\mathbf a}\colon x\mapsto \prod_{\alpha\in A} x_\alpha^{d_\alpha}\prod_{\beta\in B} y_\beta^{-e_{\beta}} f_{\mathbf a} \]
is regular in codimension~$1$, hence regular, on~$\bar P\times \mathscr U_A$, 
and it does not vanish on~$P\times \mathscr U_A$.
Moreover, $u_{\mathbf a}$ is homogeneous of degree~$1$ in~$\mathbf a$.
Consequently, and maybe up to shrinking~$\mathscr U_A$ a little bit, there exists
a real number~$\kappa$ such that $u_{\mathbf a}$  
admits uniform lower-  and upper-bounds
of the form
\[ \norm{\mathbf a} d([\mathbf a],\partial P)^{\kappa} \ll \abs{u_{\mathbf a}(\mathbf x)} 
\ll \norm {\mathbf a} , \]
for $\mathbf x\in \mathscr U_A$ and $\mathbf a\in P$.

Our integrals
now take the form
\begin{eqnarray*}
\int_{\mathscr U_A} \prod_{\alpha\in A} \abs{x_\alpha}^{\lambda_\alpha s-\rho_\alpha} 
  \prod_{\beta\in B} \abs{y_\beta}^{\lambda_\beta(s)} 
  \psi_v(u_{\mathbf a} \prod x_{\alpha}^{-d_\alpha} \prod y_\beta^{e_\beta}) \theta_A(\mathbf x;\mathbf s)\,
\\
\prod_{\alpha\in A}\mathrm dx_\alpha\, \prod_{\beta\in B}\mathrm dy_\beta
\, \prod_{\gamma\in C} \mathrm dz_\gamma. 
\end{eqnarray*}

If some of the $d_\alpha$ are zero, these variables do not
appear in the argument of the character~$\psi$. 
We first integrate with respect to these
variables, applying the method used for the trivial character 
to obtain a meromorphic continuation with poles at most
as in the product $\prod_{d_\alpha=0} \zeta_{F_\alpha}(\lambda_\alpha s-\rho_\alpha+1)$.
Let 
\[ \sigma=\max_{d_\alpha=0} (\rho_\alpha-1)/\lambda_\alpha \]
be the abscissa of the largest pole of this product.

We now  explain why having some positive $d_\alpha$
implies \emph{holomorphic} continuation with respect 
to the corresponding variables~$s_\alpha$.
The discussion distinguishes two cases.
\begin{enumerate}
\item 
\emph{All $e_\beta$ are~$\leq 0$.}

In this case, 
we first integrate with respect to a variable~$x_{\alpha_0}$ such 
that $d_{\alpha_0}>0$ and use Lemma~\ref{lemm.irregular}.
This shows that our integral is equal to another integral of the form
\[
\int
  \prod_{\substack{\alpha\in A\\ \alpha\neq\alpha_0}}
              \abs{x_{\alpha}}^{\lambda_{\alpha}s-\rho_\alpha}
  \prod_{\beta\in B} \abs{y_\beta}^{\lambda_\beta(s)}
  F(\mathbf x';\mathbf a;\mathbf s)
\prod_{\substack{\alpha\in A\\\alpha\neq\alpha_0}}
     \mathrm dx_\alpha\, \prod_{\beta\in B}\mathrm dy_\beta 
        \, \prod_{\gamma\in C} \mathrm dz_\gamma .\]
Here, $F$ is some Schwartz function
of the parameters
$$
\mathbf x'
=((x_\alpha)_{\alpha\neq\alpha_0}; (y_\beta)_{\beta\in B};
(z_\gamma)_{\gamma\in C}),
$$ 
which is continuous in~$\mathbf a$
and~$\mathbf s$ and such that  $\abs{F(\mathbf x';\mathbf a;\mathbf s)}$ is bounded by
\[  
 \ll\left( \norm{ {\mathbf a} }
 \prod_{\substack{\alpha\in A \\ \alpha\neq\alpha_0}} \abs{x_\alpha}^{-d_\alpha} \prod_{\beta\in B} \abs{y_\beta}^{e_\beta}
\right)^{-1/{d_{\alpha_0}}} \!\!\!\!\!\! \!\!\!\!\!\! \!\!\!
 \\ \norm{\mathbf a}^{-1/d_{\alpha_0}}
   \prod_{\alpha\neq\alpha_0} \abs{x_\alpha}^{d_\alpha/d_{\alpha_0}}
   \prod_{\beta} \abs{y_\beta}^{e_\beta/d_{\alpha_0}}.  
\]
Consequently, 
the exponent of all variables $x_\alpha$, for $\alpha\neq\alpha_0$,
has been increased
by the positive quantity $d_{\alpha}/d_{\alpha_0}$,
leading to absolute convergence in a neighborhood of~$\sigma$.

\item
\emph{Some $e_\beta$ is positive.}

The strategy here is to first integrate
with respect to these variables $y_\beta$ such that~$e_\beta > 0$. 
Let us apply to the inner integral $G(\mathbf x;\mathbf a;\mathbf s)$
the bound for oscillatory integrals 
that we established in Proposition~\ref{lemm.weylII-s};
we obtain an upper-bound of the form 
\[ \norm{\mathbf a}^{-\kappa} \prod_\alpha \abs{x_\alpha}^{\kappa d_\alpha}, \]
where $\kappa$ is a positive real number.
This implies that our initial integral can be computed as
\[ \int \prod_{\alpha\in A} \abs{x_\alpha}^{\lambda_\alpha s-\rho_\alpha}
 G(\mathbf x;\mathbf a;\mathbf s) \,\prod_{\alpha\in A}\mathrm dx_\alpha. \]
The gain~$\kappa d_\alpha$ in the exponent of~$\abs{x_\alpha}$
is enough to insure absolute convergence in a neighborhood
of~$\sigma$, as claimed.
\end{enumerate}
In both cases, there exists a positive real number~$\delta$
such that the remaining integral converges absolutely, and uniformly, 
to a holomorphic function on the half-plane $\Re(s)>\sigma-\delta$,
with polynomial decay in terms of~$\norm{\mathbf a}$.
As a consequence, we obtain:

\begin{prop}\label{prop.non-trivial-S}
Let $\mathbf a$ be a non-trivial character and $v\in S$. 
Then there exist holomorphic functions  $\phi_{A,v}$, defined
for $\Re(s)>\sigma-\delta$, such that
\[ \hat H_v(\mathbf a;s\lambda) = \sum_{\substack{A\in\Clan_{F_v}(X\setminus G)
\\ d_\alpha(\mathbf a)=0\, \forall\alpha\in A \\ \text{$A$ maximal}}}
\phi_{A,v}(\mathbf a;s) \prod_{\tilde\alpha=(\alpha,w)\in A} \zeta_{F_{\tilde\alpha}}( \lambda_\alpha s-\rho_\alpha+1).
\]
Moreover, $\phi_{A,v}$ satisfy upper-bounds
of the form
\[ \abs{\phi_{A,v}(\mathbf a;s)} \ll  
  \frac{ (1+\abs s)^\kappa }{ d_v([\mathbf a],\partial P_{\mathbf a})^\kappa  \norm{\mathbf a}_v^{\kappa'}}, \]
where $P_{\mathbf a}$ is the stratum of~$\mathbf P^{n-1}$
containing the line~$[\mathbf a]$, $d_v(([\mathbf a],\partial P_{\mathbf a})$ 
the $v$-adic distance of~$[\mathbf a]$ to the boundary
$\partial P_{\mathbf a}=\overline{P_{\mathbf a}}\setminus P_{\mathbf a}$,
and $\kappa,\kappa'$ positive absolute constants.
\end{prop}
Note that in this statement, 
the subsets~$A$ over which the decomposition of~$\hat H_v$ runs
are the maximal faces of the sub-complex of the analytic
Clemens complex (depending on~$\mathbf a$)
given by the vanishing of the coefficients~$d_\alpha(\mathbf a)$.
These faces need not all have the same dimension.

\subsection{Application of the Poisson summation formula}

  \subsubsection{Meromorphic continuation of the height zeta function}
From now on, we consider the case where the height is attached
to the log-anticanonical line bundle~$-(K_X+D)$
with $D=X\setminus U$. Let 
$$
\lambda=\rho-\sum_{\alpha\in\mathscr A_D} D_\alpha
$$
be the corresponding class. In that case, 
the abscissa of convergence of the height zeta function is $\sigma=1$.

Recall that we have decomposed the projective space of 
non-trivial characters~$\mathbf a$ (modulo scalars)
as a disjoint union of locally closed strata~$(P_i)$ and that
on each stratum we obtained a 
uniform meromorphic continuation of $\hat H(\mathbf a;\mathbf s)$.
The height zeta function 
can then be expressed as
\[ 
\mathrm Z(s)= \mathrm Z_0(s) + \sum_P \mathrm Z_P(s)
\] 
where  $\mathrm Z_0(s)=\hat H(0;s\lambda)$ is the term of the right hand side
of the Poisson formula~\eqref{eqn:poisson} corresponding to the trivial character, 
the summation is over all strata,
and the contribution from stratum~$P$ is
\[ 
\mathrm Z_P(s)= \sum_{\substack{\mathbf a\neq 0\\ [\mathbf a]\in P}}
        \hat H({\mathbf a};s\lambda). \]
The following lemma implies that each of $\mathrm Z_P$ 
converges absolutely for $\Re(s)>1$ and admits a meromorphic
continuation to $\Re(s)>1-\delta$, for some $\delta >0$. 

\begin{lemm}
\label{lemm.distance}
Let $Z$ be a closed subvariety of a projective space~$\mathbf P^{n-1}$
over the number field~$F$. 
For each place~$v\in S$, let $d_v(\cdot,Z)$
denote the $v$-adic distance of a point in~$\mathbf P^{n-1}(F_v)$ to 
the subset $Z(F_v)$.

Let $\mathfrak d$ be an $\mathfrak o_F$-lattice in~$F^{n}$ and
let $\norm{\cdot}_\infty$ be any norm on the real vector space
$F^n\otimes_\Q\R$.
Then, there are positive constants~$C$ and~$\kappa$ such that
\[ \prod_{v\in S} d_v([\mathbf a],Z) \geq
        C    (1+\norm{\mathbf a}_\infty )^{-\kappa} \]
for any $\mathbf a\in \mathfrak d\setminus\{0\}$ such that $[\mathbf a]\not\in Z(F)$.
\end{lemm}

\begin{proof}
Let $\Phi$ be a family of homogeneous polynomials 
defining~$Z$ set-theoretically in~$\mathbf P^{n-1}$. 
Let~$\tilde Z$ be the cone over~$Z$ in~$\mathbf A^n$
defined by the same polynomials.
We may assume that they have coefficients in~$\mathfrak o_F$
and that all their degrees are equal to an integer~$d$.
Let $v\in S$;
by the $v$-adic version of \L ojasiewicz inequalities, 
the $v$-adic distance of a point~$[\mathbf a]\in\P^{n-1}(F)$
to~$Z(F_v)$  satisfies 
\[  \log d_v([\mathbf a],Z) \approx \max_{\phi\in\Phi} \log \norm{\phi}_v([\mathbf a]), \]
where
\[ \norm{\phi}_v([\mathbf a]) = \frac{\abs{\phi(\mathbf a)}_v}{\norm {\mathbf a}_v^d}. \]
For our purposes, we thus may replace the distance~$d_v$
by the function $\max_{\phi\in\Phi}\norm{\phi}_v$.
It follows from the product formula that when $\mathbf a$
runs over all elements of~$F^n$ not in~$\tilde Z$,
then
\[ \prod_{w\in\Val(F)} \max_{\phi\in\Phi} \abs{\phi(\mathbf a)}_w \geq 1. \]

When $\mathbf a$ belongs to~$\mathfrak d$ and $w$ is a finite place of~$F$
then $\abs{\phi(\mathbf a)}_w$ is bounded from above, by a constant~$c_w$
which may be chosen equal to~$1$ for almost all~$w$. 
Since $S$ contains all archimedean places of~$F$, 
\[ \prod_{v\in S} \max_{\phi\in\Phi} \abs{\phi(\mathbf a)}_v
         \geq c_1= 1/\prod_{w\not\in S} c_w. \]
Consequently, for all $\mathbf a\in\mathfrak d$ outside~$\tilde Z$,
\[  \prod_{v\in S} \max_{\phi\in\Phi} \norm{\phi}_v([\mathbf a])
= \prod_{v\in S} \max_{\phi\in\Phi}  \frac{\abs{\phi(\mathbf a)}_v }{\norm {\mathbf a}_v^d}
\geq \frac{c_1}{ \prod_{v\in S}\norm{\mathbf a}_v^d}
\geq \frac{c_2}{\prod_{v\mid\infty}\norm{\mathbf a}_v^d},  \]
since, for any finite place $v\in S$,
$\norm{\mathbf a}_v$ is bounded from above on the lattice~$\mathfrak o_F$.
For any archimedean place~$v\in F$, let $e_v=1$ if $F_v=\R$
and $e_v=2$ if $F_v=\C$. Then, 
using the inequality between geometric and arithmetic means, we have
\[ \prod_{v\mid\infty} \norm{\mathbf a}_v
 = 1 \cdot \prod_{v\mid\infty} \left(\norm{\mathbf a}_v^{1/e_v}\right)^{e_v}
\leq \left( \frac{1 + \sum_{v\mid\infty} e_v 
\norm{\mathbf a}_v^{1/e_v}}{1+[F:\Q]}\right)^{1+[F:\Q]}. \]
Now, given the equivalence of norms on the real vector space~$F^n\otimes_\Q\R$
we may assume that $\norm{\cdot}_\infty=\sum_{v\mid\infty} e_v \norm{\mathbf a}_v^{1/e_v}$.
This implies that there exists a positive real number~$c_3$ such that
\[ \prod_{v\mid\infty} \norm{\mathbf a}_v \leq c_3 (1+\norm{\mathbf a}_\infty)^{1+[F:\Q]}, \]
and finally
\[  \prod_{v\in S} \max_{\phi\in\Phi} \norm{\phi}_v([\mathbf a]) 
\geq (c_2/c_3) \left(1+\norm{\mathbf a}_\infty\right)^{-\kappa}, \]
with $\kappa=d(1+[F:\Q])$. The lemma is proved.
\end{proof}

Let us fix a stratum~$P$. Let $\mathbf A$ be
the set of all families $A=(A_v)_{v\in S}$, where, for each~$v\in S$,
$A_v$ is a maximal subset of~$\mathscr A$ such that
$D_{A_v}(F_v)\neq\emptyset$, and $d_\alpha=0$ for all
$\alpha\in A_v$ on the stratum~$P$. (Recall that by construction,
$d_\alpha$ is constant on each stratum.)
By Proposition~\ref{prop.non-trivial-S} and Corollary~\ref{coro.otherchars},
combined with the results of Section~\ref{sec:non-trivial-varch},
for each family $A=(A_v)$ in~$\mathbf A$
and each $\mathbf a$ in the stratum~$P$,
there exists
a holomorphic function~$\phi_{A}(\mathbf a;\cdot)$ on the half-plane
$\Re(s)>1-\delta$ such that
\[ \hat H(\mathbf a;s\lambda) = \sum_{A\in\mathbf A}
   \phi_A (\mathbf a;s) \prod_{\alpha\in\mathscr A_0^D(\mathbf a)}\zeta^S_{F_\alpha}(1+\lambda_\alpha s-\rho_\alpha)
 \prod_{v\in S}\prod_{\tilde\alpha\in A_v} \zeta_{F_{\tilde\alpha}}(\lambda_\alpha s-\rho_\alpha).
\]
Moreover, this function~$\phi_{A}(\mathbf a;s)$ satisfies
estimates of the form
\[ \abs{\phi_A(\mathbf a;s)} \ll (1+\norm a_\infty)^\eps (1+\abs s)^{\kappa+\sum \kappa'_v}
 \prod_{v\in S} d_v([\mathbf a],\partial P)^{-\kappa}  \prod_{v\in S}\norm{\mathbf a}^{-\kappa'_v}, \]
where $\kappa$ and $\kappa'_v$ are positive real numbers,
and $\kappa'_v$ can be chosen arbitrarily large if $v$ is archimedean.
The series $\mathrm Z_P$ decomposes as a sum, for $A\in\mathbf A$,
of subseries:
\[ \mathrm Z_P (s)= 
 \prod_{\alpha\in\mathscr A_0^D(\mathbf a)}\zeta^S_{F_\alpha}(1+\lambda_\alpha s-\rho_\alpha)
 \prod_{v\in S}\prod_{\tilde\alpha\in A_v} \zeta_{F_{\tilde\alpha}}(\lambda_\alpha s-\rho_\alpha)
 \mathrm Z_{P,A}(s), \]
where 
\[ \mathrm Z_{P,A}(s)= \sum_{\mathbf a\in P} \phi_A(\mathbf a;s). \]

Taking $\kappa'_v$ large enough for $v$ archimedean,
Lemma~\ref{lemm.distance} implies that 
the series $\mathrm Z_{P,A}$ is bounded term by term 
by the convergent series
\[\sum_u \frac1{(1+\norm{u})^N}, \]
where  $u$ runs over a lattice in the vector space $F^n\otimes_\Q\R$
and $N$ is an arbitrarily large integer, at a cost of $(1+\abs s)^N$.
This provides the desired meromorphic continuation of~$\mathrm Z_P$
to $\Re(s)>1-\delta$ and, consequently, of the height zeta function~$\mathrm Z$.

\subsubsection{Leading poles (log-anticanonical line bundle)}\label{sec.leading-pole}
For any character~$\mathbf a$, let $b_{\mathbf a}$ be the order
of the pole of the subseries in the Fourier expansion
of the height zeta function
corresponding  to characters collinear to~$\mathbf a$.

The order of the pole of the term corresponding to the trivial
character is
\[ b_{0}= \Card(\mathscr A\setminus \mathscr A_D)+\sum_{v\in S} 
\max_{\substack{B\subset  \mathscr A_{D,v}
\\ D_B(F_v)\neq\emptyset }} \Card B, \]
while, for $\mathbf a\neq 0$, one has
\[ b_{\mathbf a}\leq \#\{\alpha \in (\mathscr A\setminus \mathscr A_D)\,;\, d_\alpha(\mathbf a)=0\}
+ \sum_{v\in S} \max_{\substack{B\subset \mathscr A_{D,v}
\\ D_B(F_v)\neq\emptyset }} \Card \{ \alpha\in B\,;\, d_\alpha(\mathbf a)=0\}. \]

\begin{lemm} 
\label{lemm.leading-pole-r}
For any nonzero character~$\mathbf a$, $b_{\mathbf a}<b_0$.
\end{lemm}

\begin{proof}
By contradiction. Assume that $b_{\mathbf a}=b_0$.
Comparing the formulae for~$b_{\mathbf a}$ and~$b_0$,
we see that
\begin{enumerate}
\item $d_\alpha(\mathbf a)=0$ for any $\alpha\in\mathscr A_D$
\item for any ~$v\in S$, there exists a subset 
$B\subset \mathscr A_{D,v}$ of maximal cardinality such that 
$D_B(F_v)\neq\emptyset$; moreover, 
$d_\alpha(\mathbf a)=0$ for all $\alpha\in B$.
\end{enumerate}
Fix a $\mathbf y\in G(F)$ such that $\langle \mathbf a,\mathbf y\rangle=1$.
Let us fix a $v\in S$ and let $B$ be a subset as in Condition~b).
By definition, 
the rational function $f_{\mathbf a}=\langle\mathbf a,\cdot\rangle$
is defined and nonzero at the generic point of the stratum~$D_B$.
Moreover, there exists such a generic point~$\mathbf x$ in $D_B(F_v)$, by the
choice of~$B$. Let $\mathbf x_\infty=\lim_{t\ra\infty} t \mathbf y\cdot \mathbf x$; this
is a point of $D_B(F_v)$.
The rational function~$t\mapsto f_{\mathbf a}(t \mathbf y\cdot \mathbf x)$ is well
defined on~$\mathbf P^1$, and
$f_{\mathbf a}(t \mathbf y\cdot \mathbf x)\ra\infty$ when $t\ra\infty$.
Consequently, the limiting point~$\mathbf x_\infty$  
belongs to a divisor~$D_\alpha$
such that $d_\alpha(\mathbf a)>0$. 
By condition~a), $\alpha\not\in \mathscr A_D$, 
hence the stratum $B'=B\cup\{\alpha\}$ is contained
in~$D=X\setminus U$, violating condition~b).
\end{proof}


\subsubsection{Application of a Tauberian theorem}
We have shown that the height zeta function~$\mathrm Z$ admits a 
meromorphic continuation to $\Re(s)>1-\delta$.
The poles are all on the line $\Re(s)=1$; with the exception
of the pole at $s=1$,
they are all given by the local factor at finite places in~$S$
of the zeta function of the fields~$F_\alpha$, and with order
\[ \sum_{\substack{v\in S \\ \text{$v$ finite}}}
     \max_{\substack{B\subset\mathscr A_{D,v} \\ D_B(F_v)\neq\emptyset}} \Card B. \]
At $s=1$, there may be a supplementary pole caused by the local
factors at archimedean places, and by the Euler product at places outside~$S$.

By the preceding lemma, $\mathrm Z$ has a pole of highest order
at $s=1$, contributed by the Fourier transform at the trivial
character only. Consequently,
\[ \lim_{s\ra 1} (s-1)^{b_0}\mathrm Z(s) =
 \lim_{s\ra 1} (s-1)^{b_0} \hat H(0;s\lambda ). \]
Let us write $\hat H^S(0;s\lambda)=\prod_{v\not\in S} \hat H_v(0;s\lambda)$.
According to~\cite{chambert-loir-tschinkel2009a}, Proposition~4.12,
\[  \lim_{s\ra 1} (s-1)^{\rang(\Pic(U))} \hat H^S(0;s\lambda)
        = \prod_{\alpha\not\in\mathscr A_D}\frac 1{\rho_\alpha}
\cdot
         \int_{U(\AD_F^S)} \prod_{v\not\in S}\delta_v(\mathbf x)
\cdot
          \mathrm d\tau^S_{(X,D)}(\mathbf x). \]
Observe that the integral in this formula is the volume,
with respect to the measure~$\tau_{(X,D)}^S$, of the 
set of $S$-adelic integral points in $X(\AD_F^S)$.

Moreover, for any place $v\in S$,
Theorem~4.3 of~\cite{chambert-loir-tschinkel2009a} shows that
that the limit
$$
\lim_{s\ra 1} (s-1)^{1+\dim\Clan_{F_v}(D)} \hat{H}_v(0; s\lambda) 
$$
is given by
$$
\sum_{\substack{A\subset\Clan_{F_v}(D)\\ \dim A=\dim\Clan_{F_v}(D)}} 
\prod_{\alpha\in A} \frac1{\rho_\alpha-1}  
\int_{D_A(F_v)} \prod_{\alpha\in\mathscr A_D\setminus A} 
\norm{\mathsec f_\alpha}_v(\mathbf x)^{-1}\,\mathrm d\tau_{D_A}(\mathbf x). 
$$
We have
\[  b_0=\rang(\Pic(U))+\sum_{v\in S} (1+\dim \Clan_{F_v}(D)). \]
Other poles of $\hat H(0;\lambda s)$ on the line $\Re(s)=1$ 
are contained in finitely many arithmetic progressions,
contributed by finitely many local factors of Dedekind zeta functions
for $v\in S$. Combining local and global estimates
in vertical strips, as in Sections 4.1
and 4.3 of~\cite{chambert-loir-tschinkel2009a},
and applying the Tauberian theorem~A.1 of that reference,
we finally obtain our Main theorem.

\begin{theo}\label{theo.main}
Let $X$ be a smooth projective equivariant compactifications 
of a vector group $G=\mathbb G_a^n$ over a number field $F$.  
Assume that $X\setminus G$ is geometrically 
a strict normal crossing divisor.    

Let $U\subset X$ be a dense $G$-invariant open subset 
and $D= X\setminus U$ the boundary. Fix
a smooth adelic metric on the log-anticanonical divisor~$-(K_X+D)$.
Let $S$ be a finite set of places of~$F$ containing the archimedean places.

Let $\mathscr U(\mathfrak o_{F,S})$ be 
the set of $S$-integral points on an integral model~$\mathscr U$ of~$U$
as in Section~\ref{sect:integral-points}. 
Let $N(B)$ be the number of points in $G(F)\cap \mathscr U(\mathfrak o_{F,S})$ 
of log-anticanonical height $\leq B$, as in Section~\ref{sect:height}.
Let 
\[ b=\rang(\EP(U)) + \sum_{v\in S} (1+\dim\Clan_{F_v}(D)). \]
and
\[ c = 
 \prod_{\alpha\not\in \mathscr A_D}\frac1{\rho_\alpha}
\cdot \tau_{(X,D)}^S(U(\AD_F^S)^{\text{\upshape int}}) \cdot
  \prod_{v\in S}\tau_v^{\max}(D(F_v)), \]
where 
\[
\tau_{(X,D)}^S(U(\AD_F^S)^{\text{\upshape int}})
= \int_{U(\AD_F^S)} \prod_{v\not\in S} \delta_v(\mathbf x)\cdot
           \mathrm d\tau_{(X,D)}^S(\mathbf x)\]
while, for each $v\in S$, $\tau_v^{\max}$ is the measure on~$D(F_v)$
given by
\[ \mathrm d\tau_v^{\max} (x)= 
\sum_{\substack{A\subset\Clan_{F_v}(D)\\ \dim A=\dim\Clan_{F_v}(D)}}
\prod_{\alpha\in A} \frac1{\rho_\alpha-1}  \prod_{\alpha\in\mathscr A_{D,v}\setminus A} \norm{\mathsec f_\alpha}_v(\mathbf x)^{-1}\,\mathrm d\tau_{D_A}(\mathbf x). \]

Then, as $B\ra\infty$, $N(B)$ has the following asymptotic expansion:
\[ 
N(B) = c\,  B (\log B)^{b-1}  
(1+\mathrm O(1/\log B)). \]
\end{theo}

\begin{rema}\label{rema:N-V}
Let $V(B)$ be the volume, with respect to the chosen Haar measure
on~$G(\AD_F)$ of the set of $S$-integral adelic points of
log-anticanonical height $\leq B$.
By definition, $\hat H(0;s\lambda)$ is the Stieltjes--Mellin transform
of the measure~$\mathrm dV(B)$ on~$\R_+$.
By the same Tauberian theorem, one has
\[ N(B)\sim V(B). \]
\end{rema}

\subsubsection{Equidistribution}\label{sec:equidistribution}
Note that all considerations above apply to an arbitrary 
smooth adelic metrization of the log-anticanonical line bundle.

Applying the abstract equidistribution theorem 
(Proposition~2.10 of~\cite{chambert-loir-tschinkel2009a}), 
we obtain that integral points of height~$\leq B$
equidistribute, when $B\ra\infty$, towards the unique
probability measure proportional to
\[ \prod_{v\not\in S}\delta_v(x)\cdot \mathrm d\tau_{(X,D)}^S (x)\cdot \prod_{v\in S}\mathrm d\tau_v^{\max}(x).\]

\def\noop#1{\ignorespaces}

\bibliographystyle{smfplain}

\bibliography{short}

\providecommand{\noopsort}[1]{}\providecommand{\url}[1]{\textit{#1}}
\providecommand{\bysame}{\leavevmode ---\ }
\providecommand{\og}{``}
\providecommand{\fg}{''}
\providecommand{\smfandname}{\&}
\providecommand{\smfedsname}{\'eds.}
\providecommand{\smfedname}{\'ed.}
\providecommand{\smfmastersthesisname}{M\'emoire}
\providecommand{\smfphdthesisname}{Th\`ese}
\begin{thebibliography}{10}

\bibitem{batyrev-t95b}
{\scshape V.~V. Batyrev {\normalfont \smfandname} {\relax Yu}.~Tschinkel} --
  {\og Rational points on bounded height on compactifications of anisotropic
  tori\fg}, \emph{Internat. Math. Res. Notices} \textbf{12} (1995),
  p.~591--635.

\bibitem{batyrev-t98}
\bysame , {\og Tamagawa numbers of polarized algebraic varieties\fg}, in
  \emph{Nombre et r{\'e}partition des points de hauteur born{\'e}e} (E.~Peyre,
  \smfedname), Ast{\'e}risque, no. 251, 1998, p.~299--340.

\bibitem{chambert-loir-t2002}
{\scshape A.~Chambert-Loir {\normalfont \smfandname} {\relax Yu}.~Tschinkel} --
  {\og On the distribution of points of bounded height on equivariant
  compactifications of vector groups\fg}, \emph{Invent. Math.} \textbf{148}
  (2002), p.~421--452.

\bibitem{chambert-loir-tschinkel2009a}
\bysame , {\og Igusa integrals and volume asymptotics in analytic and adelic
  geometry\fg}, \emph{Confluentes Mathematici} \textbf{2} (2010), p.~351--429,
  arXiv:0909.1568.

\bibitem{cluckers2011}
{\scshape R.~Cluckers} -- {\og Analytic van der {C}orput lemma for $p$-adic and
  $\mathbf {F}_q((t))$ oscillatory integrals, singular {F}ourier transforms,
  and restriction theorems\fg}, arXiv:1001.2013, 2010.

\bibitem{derenthal-loughran2009}
{\scshape U.~Derenthal {\normalfont \smfandname} D.~Loughran} -- {\og Singular
  {D}el {P}ezzo surfaces that are equivariant compactifications\fg},
  arXiv\string:0910.2717, 2009.

\bibitem{franke-m-t89}
{\scshape J.~Franke, {\relax Yu}.~I. Manin {\normalfont \smfandname} {\relax
  Yu}.~Tschinkel} -- {\og Rational points of bounded height on {F}ano
  varieties\fg}, \emph{Invent. Math.} \textbf{95} (1989), no.~2, p.~421--435.

\bibitem{hassett-t99}
{\scshape B.~Hassett {\normalfont \smfandname} {\relax Yu}.~Tschinkel} -- {\og
  Geometry of equivariant compactifications of {$\mathbf G_a^n$}\fg},
  \emph{Internat. Math. Res. Notices} \textbf{22} (1999), p.~1211--1230.

\bibitem{lachaud1982}
{\scshape G.~Lachaud} -- {\og Une pr\'esentation ad\'elique de la s\'erie
  singuli\`ere et du probl\`eme de {W}aring\fg}, \emph{Enseign. Math. (2)}
  \textbf{28} (1982), no.~1-2, p.~139--169.

\bibitem{peyre95}
{\scshape E.~Peyre} -- {\og Hauteurs et mesures de {T}amagawa sur les
  vari{\'e}t{\'e}s de {F}ano\fg}, \emph{Duke Math. J.} \textbf{79} (1995),
  p.~101--218.

\bibitem{tate67b}
{\scshape J.~T. Tate} -- {\og Fourier analysis in number fields, and {H}ecke's
  zeta-functions\fg}, in \emph{Algebraic Number Theory} (Proc. Instructional
  Conf., Brighton, 1965), Thompson, Washington, D.C., 1967, p.~305--347.

\bibitem{tate67}
{\scshape J.~Tate} -- {\og $p$-divisible groups\fg}, in \emph{Proceedings of a
  Conference on Local Fields} (Driebergen), 1967, p.~158--183.

\bibitem{weil1971}
{\scshape A.~Weil} -- \emph{Dirichlet series and automorphic forms ---
  {L}ezioni {F}ermiane}, Lecture Notes in Math., no. 189, Springer-Verlag,
  1971.

\bibitem{weil82}
{\scshape A.~Weil} -- \emph{Adeles and algebraic groups}, Progr. Math., no.~23,
  Birkh{\"a}user, 1982.

\end{thebibliography}

\end{document}